\documentclass[11pt, reqno]{amsart}

\usepackage{amsmath}
\usepackage{amssymb}
\usepackage{amsthm}
\usepackage{mathrsfs}
\usepackage{bm}
\usepackage{comment}
\usepackage{color}
\usepackage{mathtools}
\usepackage{listings}
\usepackage{iftex}

	\usepackage
	{hyperref}
	\hypersetup{colorlinks=true,citecolor=blue,linkcolor=blue,urlcolor=blue,
	pdfstartview=FitH }

\mathtoolsset{showonlyrefs=true}

\allowdisplaybreaks[3]
\theoremstyle{plain} 
\newtheorem{theorem}{Theorem}
\newtheorem{lemma}{Lemma}[section]

\newtheorem{proposition}{Proposition}

\newtheorem*{conjecture*}{Conjecture}
\newtheorem*{theorem*}{Theorem}
\newtheorem*{question*}{Question}

\theoremstyle{plain}

\theoremstyle{remark}
\newtheorem{remark}{Remark}

\theoremstyle{definition}
\newtheorem*{assumption*}{Assumption}

\newtheorem*{notations*}{Notations}
\newtheorem*{acknowledgment*}{Acknowledgments}

\numberwithin{equation}{section}

\renewcommand{\pmod}[1]{\,(\textup{mod}\,#1)}

\newcommand\swapcommand[2]{%
\let\swaptemp#1
\let#1#2
\let#2\swaptemp
}

\let\sl\l
\renewcommand\l{%
	\leavevmode
  \ifmmode
    \left
  \else
    \sl
  \fi
}

\let\sL\L
\renewcommand\L{%
	\leavevmode
	\ifmmode
		\mathscr{L}
	\else
		\sL
	\fi
}

\newcommand\set[2]{%
	\left\{ #1 \; : \; #2 \right\}
}

\newcommand\EXP[1]{
	\mathbb{E}\left[ #1 \right]
}

\swapcommand{\SS}{\S}
\renewcommand{\S}{\mathscr{S}}

\DeclareMathOperator*{\sums}{\sideset{}{^{*}}\sum}

\newcommand{\CC}{\mathbb{C}}
\newcommand{\RR}{\mathbb{R}}
\newcommand{\QQ}{\mathbb{Q}}
\newcommand{\ZZ}{\mathbb{Z}}

\newcommand{\PP}{\mathbb{P}}
\newcommand{\e}{\varepsilon}
\newcommand{\s}{\sigma}

\newcommand{\Lam}{\Lambda}

\newcommand{\qqquad}{\qquad \qquad \qquad}
\newcommand{\qqqquad}{\qquad \qquad \qquad \qquad}

\newcommand{\ol}{\overline}

\newcommand{\Res}{\operatorname{Res}}
\newcommand{\li}{\operatorname{li}}

\newcommand{\meas}{\operatorname{meas}}

\newcommand{\ceq}{\coloneqq}
\newcommand{\eqc}{\eqqcolon}

\renewcommand{\a}{\alpha}

\renewcommand{\r}{\right}

\renewcommand{\Re}{\operatorname{Re}}

\renewcommand{\epsilon}{\varepsilon}

\renewcommand{\bar}{\overline}
\renewcommand{\tilde}{\widetilde}

\title[Simultaneous large values and dependence  of Dirichlet $L$-functions]{Simultaneous large values and dependence of\\ Dirichlet $L$-functions in the critical strip}
\author[S. Inoue and J. Li]{Sh\={o}ta Inoue and Junxian Li}

\address[S. Inoue]{Department of Liberal Arts and Basic Sciences, College of Industrial Technology, Nihon University, 2-11-1 Shin-ei, Narashino, Chiba 275-8576, Japan}
\email{inoue.shota@nihon-u.ac.jp}
\address[J. Li]{Department of Mathematics, University of California, Davis, One Shields Avenue, Davis, CA 95616}
\email{njxli@ucdavis.edu}



\begin{document}

\maketitle

\begin{abstract}
	We consider the joint value distribution of Dirichlet $L$-functions in the critical strip $\frac{1}{2} < \s < 1$.
	We show that the values of distinct Dirichlet $L$-functions are \emph{dependent} in the sense that they do not behave like independently distributed random variables and they prevent each other from obtaining large values. Nevertheless, we show that distinct Dirichlet $L$-functions can achieve large values \emph{simultaneously} infinitely often.
\end{abstract}

\section{\textbf{Introduction and statements of results}}

	Many problems in number theory revolve around determining the size of $L$-functions. Let $\frac{1}{2}< \sigma<1$. Montgomery \cite{M1977} conjectured for the Riemann zeta-function $\zeta(s)$ that
	\begin{align}\label{Mconj}
		\max_{t\in [T, 2T]}\log |\zeta(\sigma+it)|\asymp \frac{(\log T)^{1-\sigma}}{(\log\log T)^{\sigma}}.
	\end{align}
	Even under the Riemann Hypothesis we only know that $\log \zeta(\sigma+it)\ll (\log t)^{2-2\sigma}/\log\log t$, which is still far from the conjecture \eqref{Mconj}. However, we do know that the lower bound in \eqref{Mconj} holds unconditionally from the work of Montgomery \cite{M1977} using Diophantine approximation, which improved the previous work of Titchmarsh \cite[Theorem 8.12]{T}. Later Aistleitner \cite{Ais} gave a different proof of the lower bound in \eqref{Mconj} using the resonance method.
	We would expect that \eqref{Mconj} holds for other $L$-functions as well, but both methods break down in establishing the lower bound in \eqref{Mconj} when the coefficients of the $L$-function are not positive. For more general $L$-functions, the best known $\Omega$-result is due to Aistleitner-Pa\'{n}kowski \cite{AP} where they showed that
	\begin{align}\label{AP}
		\max_{t\in[T, 2T]}\log | L(\sigma +it)| \gg \frac{(\log T)^{1-\sigma}}{\log \log T},
	\end{align}
	as long as $L(s)$ satisfies some suitable assumptions.
	On the other hand, one could obtain the lower bound as in \eqref{Mconj} for the Dedekind zeta-function $\zeta_K(s)$ of a number field $K$
	due to the positivity of the coefficients in the Dirichlet series (see e.g. \cite{Bala, SS}).
	The Dedekind zeta-function of a number field of degree greater than two can be factorized,
	for instance, when $K=\QQ(\eta_q)$ with $\eta_{q}$ the primitive $q$-th root of unity
	we know that $\zeta_K(s)$ can be factored as $\prod_{\chi\pmod q}L(s, \chi)$, where $\chi$ runs over all Dirichlet characters modulo $q$.
	The fact that Dedekind zeta-functions can obtain large values shows that the product of certain $L$-functions can be large,
	and thus it is natural to ask if all of these $L$-functions can achieve large values \emph{simultaneously} along vertical lines.
	The existence of simultaneous large values of general $L$-functions is an open problem mentioned in \cite{MPV}.
	We give an affirmative answer to this question in the case of Dirichlet $L$-functions.

	\begin{theorem}\label{SimulExtreme}
		Let $\frac{1}{2} < \s < 1$, $\bm \theta=(\theta_1, \dots, \theta_r)\in \mathbb{R}^r$,
		and let $\bm \chi=(\chi_1, \dots, \chi_r)$ be an $r$-tuple of distinct primitive Dirichlet characters.
		There exists some constant $c=c(\s, r) > 0$ such that for any $T \geq T_{0}(\s, \bm \chi, \bm \theta)$ sufficiently large we have
		\begin{align*}
			&\max_{t \in [T, 2T]}\min_{1 \leq j \leq r} \Re e^{-i\theta_{j}}\log{L(\s+it, \chi_j)}
			\geq c\frac{(\log{T})^{1-\s}}{\log{\log{T}}}.
		\end{align*}
	\end{theorem}
	\begin{remark}
		The lower bound agrees with the lower bound in \eqref{AP} (except the constant $c$) obtained by Aistleitner-Pa\'{n}kowski for a single general $L$-function whose coefficients are not necessarily positive.
	\end{remark}
	\begin{remark}
		For comparison, we mention the recent work of
		Mahatab-Pa\'{n}kowski-Vatwani \cite{MPV} where they used Diophantine approximation to prove existence of joint large values of a class of $L$-functions in a small neighborhood: they  showed that there exist $t_1, \dots, t_r\in [T, 2T]$ such that
		\begin{align}
			\Re e^{-i\theta_j} \log L_j(\sigma+it_j) \gg \frac{(\log T)^{1-\sigma}}{\log\log T}
		\end{align}
		with $|t_i-t_j|\leq 2(\log T)^{(1+\sigma)/2}(\log \log T)^{1/2}$
		when each of the $L$-functions $L_j$ satisfies suitable assumptions.
	\end{remark}
	To prove Theorem \ref{SimulExtreme}, we consider the joint value distribution of Dirichlet $L$-functions in the critical strip $\frac{1}{2} < \s < 1$. Let us recall the value distribution of the Riemann zeta-function.
	The study of the value distribution of $\zeta(s)$ in the critical strip dates back to the work of Bohr-Jessen \cite{BJ1930}, where a continuous limiting distribution of values of $\zeta(\s+it)$ in the case $\s > \frac{1}{2}$ was established. Precisely, for any fixed $\s > \frac{1}{2}$,
	there exists a probability measure $P_{\s}$ on $(\RR, \mathcal{B}(\RR))$ such that
	for any fixed $V \in \RR$
	\begin{align}\label{BJ}
		\Psi(T, V) \ceq \frac{1}{T}\meas\set{t \in [T, 2T]}{\log|\zeta(\s + it)| > V} \sim P_{\s}((V, +\infty)) \text{ as $T \rightarrow + \infty$.}
	\end{align}
	However, the precise behavior of $P_\sigma$ had not been determined until some 60 years later when
	Hattori-Matsumoto \cite{HM1999} first showed that  for $\frac{1}{2}< \s < 1$
	\begin{align}\label{HM}
		P_{\s}((V, +\infty))
		= \exp\l( - A(\s) V^{\frac{1}{1 - \s}} (\log{V})^{\frac{\s}{1 - \s}}(1 + o(1)) \r), \ V \rightarrow + \infty
	\end{align}
	where $A(\s)$ is defined by
	\begin{align} \label{def_G_s}
		A(\s) = \l( \frac{\s^{2\s}}{(1 - \s)^{2\s - 1} G(\s)^{\s}} \r)^{\frac{1}{1 - \s}},
		\quad  G(\s) = \int_{0}^{\infty}\frac{\log{I_{0}(u)}}{u^{1 + \frac{1}{\s}}}du.
	\end{align}
	Here $I_{0}(z) \ceq \frac{1}{2\pi} \int_{-\pi}^{\pi}\exp(z \cos \theta) d\theta = \sum_{n = 0}^{\infty}(z / 2)^{2n} / n!^2$ is the modified 0-th Bessel function.
	Later Lamzouri \cite{L2011} gave an effective bound for the $o(1)$ in \eqref{HM} uniformly
	for $ V \ll  \frac{(\log{T})^{1 - \s}}{\log{\log{T}}}$ by comparing $\zeta(s)$ with its random Euler product model.

	We would like to generalize \eqref{BJ} and \eqref{HM} for the joint value distribution of Dirichlet $L$-functions. To do this, we introduce the following notation.
	Let $\frac{1}{2} < \s < 1$. For an $r$-tuple of Dirichlet characters $\bm \chi=(\chi_1, \dots, \chi_r)$,
	$\bm{V} = (V_{1}, \dots, V_{r}) \in \RR^{r}$, and $\bm\theta \in \RR^r$ we denote
	\begin{align}
		\Psi(T, \bm V, \bm \chi, \bm{\theta})
		\ceq \frac{1}{T}\meas\set{ t \in [T, 2T]}{\Re e^{-i\theta_{j}}\log{L(\s + it, \chi_j)} > V_{j}
			\text{ for } j = 1, \dots, r}.
	\end{align}
	From \cite[Theorem 12.1]{St}, we know that there exists a probability measure $P^{\bm \chi}_\sigma$ on $(\RR^r, \mathcal B(\RR^r))$ such that
	\begin{align}
		\Psi(T, \bm V, \bm \chi, \bm{\theta})
		\sim P^{\bm{\chi}}_\sigma ((V_1, \infty) \times \cdots \times (V_{r}, \infty)), \quad T\rightarrow\infty.
	\end{align}
	The existence of the limiting distribution can be used to study the joint universality theorem (see e.g. \cite{LK}). We are interested in determining the behavior of $P^{\bm \chi}_\sigma$ when $V_j\rightarrow\infty$.
	To state the results, we define for $\bm \alpha=(\alpha_1, \dots, \alpha_r)\in (\RR_{>0})^r$
	\begin{align} \label{def_xi_chi}
		\xi(\s, \bm{\chi}, \bm{\theta}; \bm{\a})
		\ceq \frac{1}{\phi(d)} \sum_{u \in (\ZZ / d\ZZ)^{\times}}
		\bigg|\sum_{j = 1}^{r} \a_{j} e^{-i\theta_{j}} \chi_{j}(u)\bigg|^{\frac{1}{\s}}
	\end{align}
	with $\phi$ the Euler totient function and $d$ the least common multiple of the moduli of $\chi_{1}, \dots, \chi_{r}$.
	Throughout this paper, we denote the vector with all components being one by $\bm{1}\in \RR^{r}$ for some suitable positive integer $r$ which should be clear from the content and so we omit it in this notation. We say that two Dirichlet characters $\chi_{1}$ and $\chi_{2}$ are equivalent denoted by $\chi_{1} \sim \chi_{2}$ if they are induced by the same primitive character.

	We first state a result concerning the joint value distribution for two Dirichlet $L$-functions.
	\begin{theorem}\label{thm2}
		Let $\frac{1}{2} < \s < 1$ and $\bm \theta\in \RR^2$. Let $\bm{\chi} = (\chi_1,\chi_2)$ with $\chi_1\not\sim \chi_2$.
		Then, there exists a positive constant $a_1 = a_1(\s)$ such that
		for any large $T, V$ satisfying $V \leq a_1\frac{(\log{T})^{1 - \s}}{\log{\log{T}}}$, we have
		\begin{align}
			& \Psi(T,\bm V, \bm \chi, \bm \theta)
			= \exp\l(-\tilde\xi(\sigma, \bm{\chi}) A(\s) V^{\frac{1}{1-\s}} (\log{V})^{\frac{\s}{1-\s}}
			\Big(1 + O_{\s, \bm{\chi}, \bm{\theta}}\Big( \l(\frac{\log{\log{V}}}{\log{V}}\r)^{1-\s} \Big) \Big) \r)
		\end{align}
		where  $\bm{V} = (V, V)$ and
		\begin{align}
			\tilde\xi(\sigma, \bm \chi, \bm \theta)=2^{\frac{1}{1-\sigma}}\Big(\xi(\sigma, \bm \chi, \bm \theta; \bm 1)\Big)^{\frac{-\sigma}{1-\sigma}},
		\end{align}
		and $\xi(\sigma, \bm{\chi}, \bm \theta;\bm 1)$ is defined in \eqref{def_xi_chi} and $A(\sigma)$ is defined in \eqref{def_G_s}.
	\end{theorem}
	The joint value distribution of $L$-functions on the critical line $\sigma = \frac{1}{2}$ was considered by Bombieri-Hejhal \cite{BH1995}, who showed that normalized values of $L$-functions satisfying certain assumptions behave like \emph{independent} Gaussian distributed random variables.
	One may wonder if the independence still holds when we move into the critical strip away from the critical line. From work of Voronin \cite{Vor} and Lee-Nakamura-Pa\'nkowski \cite{LNP2017}, we know that different $L$-functions under suitable assumptions satisfy joint universality properties, which indicate some independence between these $L$-functions.
	However, we show that the ``independence'' breaks down when considering their value distributions and distinct Dirichlet $L$-functions ``repel" each other in the sense that they prevent each other obtaining large values simultaneously.
	\begin{theorem}\label{DP_DC}
		Let $\frac{1}{2}< \sigma <1, \bm \theta\in \mathbb{R}^2$ and $\bm \chi=(\chi_1, \chi_2)$ with $\chi_1\not \sim \chi_2$.
		Then $\tilde{\xi}(\sigma, \bm \chi, \bm \theta) \geq 2^{1/2(1 - \s)}$.
		In particular, $\tilde{\xi}(\sigma, \bm \chi, \bm \theta) > 2$.
	\end{theorem}

	Note that from work of Lamzouri, we can infer that \eqref{BJ} still holds if one replace the Riemann zeta-function by any fixed Dirichlet $L$-function.
	Thus if the values of distinct Dirichlet $L$-functions behave like independent random variables, we would have $\tilde{\xi}(\sigma, \bm \chi, \bm \theta) =2$ contrary to Theorem \ref{DP_DC}. Note that $\tilde{\xi}(\s, \bm{\chi}, \bm{\theta})$ can be arbitrarily large as $\sigma$ approaches $1$.

	The joint distribution for a generic set of $r$-tuples of Dirichlet $L$-functions is more complicated.
	We need to further introduce
	the arithmetic factors
	$\Xi_{j}(\s, \bm{\chi}, \bm{\theta}; \bm{\a})$
	for $1\leq j\leq r$, $\bm{\a} = (\a_{1}, \dots, \a_{r}) \in (\RR_{> 0})^{r}$
	\begin{align} \label{def_Xi_j}
		&\Xi_{j}(\s, \bm{\chi}, \bm{\theta}; \bm{\a})\\
		&\ceq \frac{1}{\phi(d)}\sums_{u \in (\ZZ / d \ZZ)^{\times}}
		\bigg| \sum_{\ell = 1}^{r}\a_{\ell}e^{-i\theta_{\ell}}\chi_{\ell}(u) \bigg|^{\frac{1}{\s} - 2}
		\Re\overline{e^{-i\theta_{j}}\chi_{j}(u)}\sum_{k = 1}^{r} \a_{k} e^{-i\theta_{k}} \chi_{k}(u),
	\end{align}
	where $d$ is the least common multiple of the moduli of $\chi_{1}, \dots, \chi_{r}$
	and $\displaystyle{\sideset{}{^{*}}\sum_{u \in (\ZZ / d \ZZ)^{\times}}}$ means the sum is over the irreducible residue class of $d$
	such that $\sum_{\ell = 1}^{r}\a_{\ell} e^{-i\theta_{\ell}} \chi_{\ell}(u) \not= 0$.
	Note that we have the relation
	\begin{align}	\label{Re_xX}
		\xi(\s, \bm{\chi}, \bm{\theta}; \bm{\a}) = \sum_{j = 1}^{r}\a_{j}\Xi_{j}(\s, \bm{\chi}, \bm{\theta}; \bm{\a}).
	\end{align}
	With this notation, we can state the result for the joint distribution of $r$ Dirichlet $L$-functions.

	\begin{theorem}	\label{Main_Thm_EV_R}
		Let $\frac{1}{2} < \s < 1$, $\bm{\theta} \in \RR^r$, and let $\bm \chi=(\chi_1,\dots , \chi_r)$ be an $r$-tuple of pairwise inequivalent Dirichlet characters.
		Suppose $\bm{\a} = (\a_{1}, \dots, \a_{r}) \in (\RR_{> 0})^{r}$   satisfies
		$\Xi_{j}(\s, \bm{\chi}, \bm{\theta}; \bm{\a}) > 0$ for all $1\leq j\leq r$.
		Then there exists a positive constant $a_1= a_1(\s, \bm{\a})$ such that
		for any sufficiently large $T$ and for $\bm{V} = (\Xi_{1}(\s, \bm{\chi}, \bm{\theta}; \bm{\a}) V, \dots, \Xi_{r}(\s, \bm{\chi}, \bm{\theta}; \bm{\a}) V)$
		with $V$ sufficiently large satisfying $V \leq a_1\frac{(\log{T})^{1 - \s}}{\log{\log{T}}}$, we have
		\begin{align}	\label{Thm_IND1}
			& \Psi(T, \bm V, \bm \chi, \bm \theta)
			= \exp\l(-\xi(\s, \bm{\chi}, \bm{\theta}; \bm{\a}) A(\s) V^{\frac{1}{1-\s}} (\log{V})^{\frac{\s}{1-\s}}
			\Big(1 + O_{\s, \bm{\chi}, \bm{\theta}, \bm{\a}}\Big( \Big(\frac{\log{\log{V}}}{\log{V}}\Big)^{1-\s} \Big) \Big) \r).
		\end{align}
	\end{theorem}

	\begin{remark}  \label{RMKXi}
		Here the arithmetic factors $\Xi_j(\sigma, \bm \chi, \bm \theta;\bm \alpha)$ appear naturally in the partial derivatives of the cumulant-generating function and we need their positivity when applying the saddle point method.
		To apply Theorem \ref{Main_Thm_EV_R} unconditionally,
		we need to find $\bm{\a}$ such that $\Xi_{j}(\s, \bm{\chi}, \bm \theta; \bm{\a}) > 0$ for all $1\leq j\leq r$,
		but the choice of $\bm{\a}$ depends highly on $\bm \chi, \bm \theta$, and $\sigma$.
		When $r = 1, 2$ we can simply take $\bm{\a} = \bm 1$.
		However, the choice $\bm{\a} = \bm 1$ does not work for all choices of $\bm \chi$ in the full range $1/2< \sigma<1$.
		For example, we can find $8$ characters modulo $13$ such that $\Xi_j(\sigma, \bm \chi, \bm 1) < 0$ for some $j$ when $\sigma$ is close to $1$. This is the reason that we introduce the parameter $\bm \alpha$ and our choice of $\bm \alpha$ is in Theorem \ref{GDP_DC} below.
	\end{remark}
	\begin{remark}
		Here we assume that $\chi_i\not \sim \chi_j$ for all $i\not=j$ since we allow $\theta_j$ to be arbitrary.  We could relax the conditions on $\chi_j$ if we put more restrictions on $\bm\theta$. For instance, if
		$\chi_1, \dots, \chi_r$ is the set of all Dirichlet characters modulo $d$ and $\bm \theta=(\theta, \dots, \theta)$, then one can show that $\Xi_j(\sigma, \bm \chi, \bm \theta; \bm 1)>0$ for all $j$ and the conclusion in Theorem \ref{Main_Thm_EV_R} still holds. (See the remark after Lemma \ref{RL_X_x_2}.)
	\end{remark}

	Next we give our choice of $\bm \alpha$  which can be used to verify the conditions in Theorem \ref{Main_Thm_EV_R} as well as \emph{dependence} of value distributions of distinct Dirichlet $L$-functions.

	\begin{theorem}	\label{GDP_DC}
		Let $\frac{1}{2} < \s < 1$, $r \geq 2$, $\bm{\theta} \in \RR^{r}$,
		and let $\bm{\chi}$ be an $r$-tuple of pairwise inequivalent Dirichlet characters.
		For $\bm{\a} = (\a, 1, \dots, 1)$ with $\a \geq \a_{0}( \s,r)$ sufficiently large depending only on  $\s$ and $r$,
		we have $\Xi_{j}(\s, \bm{\chi}, \bm{\theta}; \bm{\a})>0$ for all $j$ and
		\begin{align}
			\xi(\s, \bm{\chi}, \bm{\theta}; \bm{\a})
			> \sum_{j = 1}^{r}\Xi_{j}(\s, \bm{\chi}, \bm{\theta}; \bm{\a})^{\frac{1}{1 - \s}}.
		\end{align}
	\end{theorem}
	\begin{remark}
		If the values of Dirichlet $L$-functions behave like independently distributed random variables,
		then we would have $\xi(\s, \bm{\chi}, \bm{\theta}; \bm{\a})
			= \sum_{j = 1}^{r}\Xi_{j}(\s, \bm{\chi}, \bm{\theta}; \bm{\a})^{\frac{1}{1 - \s}}$ contrary to Theorem \ref{GDP_DC}.
	\end{remark}

	To prove Theorem \ref{Thm_IND1}, we consider the joint value distribution for the corresponding Dirichlet polynomials.
	Define $P_{\chi}(\s + it, X) = \sum_{p \leq X}\frac{\chi(p)}{p^{\s+it}}$ and
	$ \Psi(T, \bm{V}, X)=\Psi(T, \bm{V}, X; \bm{\chi}, \bm{\theta})$ by
	\begin{align}
		& \Psi(T, \bm{V}, X)
		\ceq \frac{1}{T}\meas \set{t \in [T, 2T]}{\Re e^{-i\theta_{j}}P_{\chi_{j}}(\s + it, X) > V_{j}
		\text{ for all $j=1, \dots, r$}}.
	\end{align}

	\begin{proposition}	\label{Main_Prop_JEV}
		Let $L \geq 2$,
		and let $\s, \bm \chi, \bm \theta, \bm{\a}$ be as in Theorem \ref{Main_Thm_EV_R}.
		There exists a positive constant $a_{2} = a_{2}(\s, L, \bm{\a})$ such that
		for any large numbers $T$, $V$, $X = (\log{T})^{L}$ with $V \leq a_{2}\frac{(\log{T})^{1-\s}}{\log{\log{T}}}$,
		we have
		\begin{align}
			\Psi(T, \bm V, X)
			= \exp\l(-\xi(\s, \bm{\chi}; \bm{\a}) A(\s) V^{\frac{1}{1-\s}} (\log{V})^{\frac{\s}{1-\s}}
			\l(1 + O_{\s, \bm{\chi}, \bm{\theta}, \bm{\a}}\l( \Big(\frac{\log{\log{V}}}{\log{V}}\Big)^{1-\s} \r) \r) \r)
		\end{align}
		for $\bm{V} = (\Xi_{1}(\s, \bm{\chi}, \bm{\theta}; \bm{\a}) V, \dots, \Xi_{r}(\s, \bm{\chi}, \bm{\theta}; \bm{\a}) V)$.
	\end{proposition}

	Now we give some heuristics on the \emph{dependence} of the joint value distribution
	in the critical strip $\frac{1}{2}<\sigma<1$ by comparing it with the situation on the critical line $\s=\frac{1}{2}$.
	Roughly speaking, the value distribution of $\log L(\sigma+it, \chi)$
	can be compared with a truncated Dirichlet series over primes with random coefficients: $\sum_{p\leq X}\frac{\chi(p)\mathcal X(p)}{p^{\sigma+it}}$
	where $\{\mathcal X(p)\}_{p}$ is a sequence of independent random variables uniformly distributed on the unit circle
	and $X$ is certain parameter that will be chosen depending on the distribution of zeros close to the line $\Re(s)=\sigma$.
	On the critical line $\sigma=\frac{1}{2}$, we need a long Dirichlet polynomial (e.g  with length  $X = t^{1/(\log_{2}{t})^{O(1)}}$)
	to determine the distribution of $\log L(\frac{1}{2}+it, \chi)$, since there are many zeros near the critical line
	and we need $\frac{1}{2}\sum_{p \leq X}\frac{|\chi(p)|^{2}}{p} \sim \frac{1}{2}\log \log t$ to capture the variance.
	In this case, the orthogonality of Dirichlet characters is enough to conclude that
	the joint value distributions of different Dirichlet $L$-functions behave like independent random variables asymptotically.
	One of the key points of the independency comes from the following asymptotic formula:
	\begin{align}	\label{KPIDCL}
		\EXP{\l( \sum_{j = 1}^{r}z_{j}\Re\sum_{p \leq X}\frac{\chi_{j}(p) \mathcal{X}(p)}{p^{1/2}} \r)^{2}}
		\sim \sum_{j = 1}^{r}z_{j}^{2}\EXP{\l(\Re\sum_{p \leq X}\frac{\chi_{j}(p) \mathcal{X}(p)}{p^{1/2}} \r)^{2}}
	\end{align}
	which holds for  \emph{any} $z_{1}, \dots, z_{r} \in \CC$ as $X \rightarrow  \infty$.
	When $\sigma>\frac{1}{2}$ we only need a short Dirichlet polynomial (e.g with  length
	(e.g.\ $X \leq (\log t)^{O(1)}$) to determine the distribution of $\log L(\s+it, \chi)$,
	since there are fewer zeros close to the line $\Res=\sigma$
	and the variance $\frac{1}{2}\sum_{p\leq X}\frac{|\chi(p)|^{2}}{p^{2\sigma}}$ converges.
	In contrast to the case of the critical line, for example, we can find \emph{some} $z_1, \dots, z_r\in \mathbb{C}$ such that
	\begin{align}
		\EXP{\l( \sum_{j = 1}^{r}z_{j}\Re\sum_{p \leq X}\frac{\chi_{j}(p) \mathcal{X}(p)}{p^{\s}} \r)^{2}}
		\not\sim \sum_{j = 1}^{r}z_{j}^{2}\EXP{\l(\Re\sum_{p \leq X}\frac{\chi_{j}(p) \mathcal{X}(p)}{p^{\s}} \r)^{2}}
	\end{align}
	as $X \rightarrow \infty$ when $\sigma>\frac{1}{2}$, which then prevent distinct Dirichlet $L$-functions from being independent. The choice for $z_{1}, \dots, z_{r}$ is related to the choice of $\bm{\a}$ in Theorem \ref{GDP_DC}.

	Finally, we remark that it is an interesting question to obtain precise results on dependence and simultaneous extreme value of Dirichlet $L$-functions on the line $\s = 1$.
	It is feasible that the strategy of the paper together with the work of Granville-Soundararajan \cite{GS} will yield results in this direction,
	though some new difficulties may occur when examining the behavior of $\Xi_{j}$.

\section{\textbf{Preliminaries}}

	From now on, we use $\{ \mathcal{X}(p) \}_{p \in \mathcal{P}}$ to denote a sequence of independent random variables on a probability space $(\Omega, \mathscr{A}, \PP)$
	uniformly distributed on the unit circle in $\CC$.

	\begin{lemma}	\label{SLL}
		Let $T \geq 5$.
		Let $\{a(p)\}$ be a complex sequence over prime numbers.
		For any $X \geq 3$, $k \in \ZZ_{\geq 1}$ such that $X^{k} \leq T$, we have
		\begin{align} \label{SLL1}
			\int_{T}^{2T}\bigg| \sum_{p \leq X}\frac{a(p)}{p^{it}} \bigg|^{2k}dt
			\ll T k! \l( \sum_{p \leq X}|a(p)|^2 \r)^{k},
		\end{align}
		and for any $X \geq 3$, $k \in \ZZ_{\geq 1}$
		\begin{align} \label{SLL2}
			\EXP{\bigg| \sum_{p \leq X}a(p)\mathcal{X}(p) \bigg|^{2k}}
			\leq k! \l( \sum_{p \leq X}|a(p)|^2 \r)^{k}.
		\end{align}
		Here, the above sums run over prime numbers, and $\EXP{}$ is the expectation.
	\end{lemma}

	\begin{proof}
		This follows from  \cite[Lemma 3.6]{IL2021} and \cite[Lemma 4.3]{IL2021}.
	\end{proof}

	\begin{lemma} \label{GRLR}
		Let $\{b_{1}(m)\}, \dots, \{ b_{n}(m) \}$ be complex sequences.
		For any $q_{1}, \dots, q_{s}$ distinct prime numbers and any $k_{1, 1}, \dots k_{1, n}, \dots,  k_{s, 1}, \dots, k_{s, n} \in \ZZ_{\geq 1}$,
		we have
		\begin{align*}
			& \frac{1}{T}\int_{T}^{2T}\prod_{\ell = 1}^{s}
			\l(\Re b_{1}(q_{\ell}^{k_{1, \ell}}) q_{\ell}^{-i t k_{1, \ell}}\r)
			\cdots \l(\Re b_{n}(q_{\ell}^{k_{n, \ell}}) q_{\ell}^{-i t k_{n, \ell}}\r)dt                               \\
			& =  \EXP{\prod_{\ell = 1}^{s}\l(\Re b_{1}(q_{\ell}^{k_{1, \ell}}) \mathcal{X}(q_{\ell})^{k_{1, \ell}}\r)
			\cdots \l(\Re b_{n}(q_{\ell}^{k_{n, \ell}}) \mathcal{X}(q_{\ell})^{k_{n, \ell}}\r)}
			+ O\l( \frac{1}{T}\prod_{\ell = 1}^{s}\prod_{j = 1}^{n}q_{\ell}^{k_{j, \ell}}|b_{j}(q_{\ell}^{k_{j, \ell}})| \r).
		\end{align*}
	\end{lemma}

	\begin{proof}
		This is \cite[Lemma 4.2]{IL2021}.
	\end{proof}

	\begin{lemma}	\label{KLST}
		Let $\chi$ be a Dirichlet character modulo $q$.
		Let $\s > \frac{1}{2}$.
		There exist positive constants $\delta_{\chi}$, $A = A(\s, \chi)$ such that
		for any $k \in \ZZ_{\geq 1}$, $3 \leq X \leq T^{1/k}$,
		\begin{align} \label{MVIPC}
			&\frac{1}{T}\int_{T}^{2T}\l| \log{L(\s + it, \chi)} - \sum_{2 \leq n \leq X}\frac{\Lam(n)\chi^{\star}(n)}{n^{\s + it} \log{n}}
			- \sum_{p \mid q} \log\l( 1 - \frac{\chi^{\star}(p)}{p^{\s + it}} \r) \r|^{2k}dt\\
			&\leq A^{k} k^{4k}T^{\delta_{\chi}(1 - 2\s)}
			+ A^{k} k^{k} \l( \frac{X^{1 - 2\s}}{\log{X}} \r)^{k}.
		\end{align}
		Here, $\Lam(n)$ is the von Mangoldt function, and $\chi^{\star}$ is the primitive character that induces $\chi$.
	\end{lemma}

	\begin{proof}
		We consider the case when $\chi$ is primitive first. In this case we know that $L(s, \chi)$ belongs to the Selberg class,
		satisfies a strong zero density estimate (see \cite[Lemma 1]{F1974}), and $\sum_{p \mid q} \log\l( 1 - \frac{\chi^{\star}(p)}{p^{s}} \r)=0$.
		Hence, we derive from \cite[Proposition 3.3]{IL2021} that
		\begin{align*}
			& \frac{1}{T}\int_{T}^{2T}\l| \log{L(\s + it, \chi)} - \sum_{2 \leq n \leq X}\frac{\Lam(n)\chi(n)}{n^{\s + it} \log{n}} \r|^{2k}dt\\
			&\leq A^{k} k^{4k}T^{\delta_{\chi}(1 - 2\s)}
			+ A^{k} k^{k} \l( \sum_{X < p \leq T^{1/k}}\frac{|\chi(p)|^{2}}{p^{2\s}} \r)^{k}
		\end{align*}
		for some constants $\delta_{\chi} > 0$, $A = A(\s, \chi) > 0$.
		Using the prime number theorem and partial summation, we see that
		\begin{align*}
			\sum_{X < p \leq Y}\frac{|\chi(p)|^2}{p^{2\s}}
			\leq \sum_{p > X}\frac{1}{p^{2\s}}
			\ll_{\s} \frac{X^{1 - 2\s}}{\log{X}},
		\end{align*}
		which completes the proof in the case of primitive characters.

		When $\chi$ is imprimitive, we use the formula
		\begin{align*}
			\log{L(s, \chi)}
			= \log{L(s, \chi^{\star})} + \sum_{p \mid q} \log\l( 1 - \frac{\chi^{\star}(p)}{p^{s}} \r),
		\end{align*}
		which together with the primitive case completes the proof.
	\end{proof}

\section{\textbf{Approximate formulae for moment generating functions}}

	In this section, we give an approximate formula for moment generating functions of an $r$-tuple of Dirichlet polynomials supported on primes, in analogue to \cite[Section 4.1]{IL2021}.
	Throughout this section, we suppose that $\bm{a}(p) = (a_{1}(p), \dots, a_{r}(p))$
	is an $r$-tuple of bounded sequences supported on prime numbers.
	We define $\| \cdot \|$ stands for the maximum norm,
	that is $\| \bm{z} \| = \max_{1 \leq j \leq r}|z_{j}|$ if $\bm{z} = (z_{1}, \dots, z_{r}) \in \CC^{r}$.
	We also write $\| a_{j} \|_{\infty} \ceq \sup_{p}|a_{j}(p)|$ and $\| \bm{a} \| = \|(\| a_{1} \|_{\infty}, \dots, \| a_{r} \|_{\infty})\|$.
	For every $\bm{z} = (z_{1}, \dots, z_{r}) \in \CC^{r}$, $\s, t \in \RR$, and any prime number $p$, we define
	\begin{align}
		\label{def_K_a}
		K_{\bm{a}}(p, \bm{z})
		\ceq \sum_{j=1}^{r}z_{j} a_{j}(p)\sum_{k=1}^{r}z_{k} \ol{a_{k}(p)},
	\end{align}
	and
	\begin{align*}
		P_{j}(\s+it, X)
		\ceq \sum_{p \leq X}\frac{a_{j}(p)}{p^{\s+it}}.
	\end{align*}
	To compute the moment generating function of $(P_{j}(\s + it, X) )_{j = 1}^{r}$, we work on a subset $\mathcal A$ where the Dirichlet polynomials do not obtain large values. Precisely,
	let $\mathcal{A} = \mathcal{A}(T, X, \bm{a})$ be
	\begin{align}	\label{def_sA_JEV}
		& \mathcal{A} =
		\bigcap_{j = 1}^{r}\set{t \in [T, 2T]}{\l|P_{j}(\s+it, X) \r|
		\leq \frac{(\log{T})^{1-\s}}{\log{\log{T}}}}.
	\end{align}

	We first show that the measure of $\mathcal{A}$ is close to $T$.

	\begin{lemma}	\label{ESAEG_JEV}
		Let $L \geq 1$.
		Let $T, X$ be large numbers with $X \leq (\log{T})^{L}$.
		Then, there exists a positive number $b_{0} = b_{0}(\s, L, \| \bm{a} \|, r)$ such that
		\begin{align*}
			\frac{1}{T}\meas([T, 2T] \setminus \mathcal{A})
			\leq \exp\l(- b_{0}\frac{\log{T}}{\log{\log{T}}} \r).
		\end{align*}
	\end{lemma}

	\begin{proof}
		Let $k$ be a positive integer.
		If $X \leq k \log{2k}$, then we see that
		\begin{align*}
			\frac{1}{T}\int_{0}^{T}\bigg| \sum_{p \leq X}\frac{a_{j}(p)}{p^{\s+it}} \bigg|^{2k}dt
			&\ll \l( \sum_{p \leq k\log{2k}}\frac{|a_{j}(p)|}{p^{\s}} \r)^{2k}
			\ll \l(\frac{C_{1} k^{1 - \sigma}}{(\log{2k})^{\s}} \r)^{2k}
		\end{align*}
		for some constant $C_{1} > 0$ depending on $\s$ and $\| a_{j} \|_{\infty}$.
		Suppose that the inequality $k \log{2k} < X$ holds.
		We then write
		\begin{align*}
			&\int_{T}^{2T}|P_{j}(\s + it, X)|^{2k}dt
			\leq 4^{k}\l( \int_{T}^{2T}\bigg| \sum_{p \leq k \log{2k}}\frac{a_{j}(p)}{p^{\s+it}} \bigg|^{2k}dt
			+ \int_{T}^{2T}\bigg| \sum_{k\log{2k} < p \leq X}\frac{a_{j}(p)}{p^{\s+it}} \bigg|^{2k}dt \r).
		\end{align*}
		By estimate \eqref{SLL1} and the prime number theorem, it holds that, for $1 \leq k \leq \frac{\log{T}}{L \log{\log{T}}}$,
		\begin{align*}
			\frac{1}{T}\int_{T}^{2T}\bigg| \sum_{k \log 2k < p \leq X}\frac{a_{j}(p)}{p^{\s+it}} \bigg|^{2k}dt
			&\ll k! \l( \sum_{k \log 2k < p \leq X}\frac{|a_{j}(p)|^{2}}{p^{2\s}} \r)^{k}
			\ll \l(\frac{ C_{2} k^{1 - \sigma} }{(\log 2k)^{\s}} \r)^{2k},
		\end{align*}
		where $C_{2}$ is a positive constant which may depend on $\s$ and $\| a_{j} \|_{\infty}$.
		Furthermore, by the prime number theorem it follows that
		\begin{align}\label{eq_FS}
			\frac{1}{T}\int_{T}^{2T}\bigg| \sum_{p \leq k\log{2k}}\frac{a_{j}(p)}{p^{\s+it}} \bigg|^{2k}dt
			&\ll \l( \sum_{p \leq k\log{2k}}\frac{|a_{j}(p)|}{p^{\s}} \r)^{2k}
			\ll \l(\frac{C_{3} k^{1 - \sigma}}{(\log{2k})^{\s}} \r)^{2k}
		\end{align}
		for some positive constant $C_{3}$ which may depend on $\s$ and $\| a_{j} \|_{\infty}$.
		Hence, we have
		\begin{align} \label{MVEP}
			\frac{1}{T}\int_{T}^{2T}| P_{j}(\s + it, X)|^{2k}dt
			\ll \l(\frac{B_{j} k^{1 - \sigma}}{(\log{2k})^{\s}} \r)^{2k}
		\end{align}
		for any $1 \leq k \leq \frac{\log{T}}{L \log{\log{T}}}$.
		Here $B_{j} = B(a_{j}, \s) \ceq \max_{1 \leq i \leq 3} C_{i}$.
		Therefore, there exist positive constants $B_{j} = B_{j}(\s, \| a_{j} \|_{\infty})$ such that
		\begin{align*}
			& \frac{1}{T}\meas\set{t \in [T, 2T]}{\l|P_{j}(\s+it, X)\r|
			> \frac{(\log{T})^{1-\s}}{\log{\log{T}}}}
			\leq \l( B_{j}\frac{k^{1-\s} \log{\log{T}}}{(\log{2k})^{\s}(\log{T})^{1-\s}} \r)^{2k}
		\end{align*}
		holds for any $2 \leq k \leq \frac{\log{T}}{L \log{\log{T}}}$.
		Hence, we have
		\begin{align*}
			\frac{1}{T}\meas([T, 2T] \setminus \mathcal{A})
			& \leq \frac{1}{T}\sum_{j = 1}^{r}\meas\set{t \in [T, 2T]}{\l|P_{j}(\s+it, X) \r|
			> \frac{(\log{T})^{1-\s}}{\log{\log{T}}}}                                                    \\
			& \leq \l( B\frac{k^{1-\s} \log{\log{T}}}{(\log{k})^{\s}(\log{T})^{1-\s}} \r)^{2k},
		\end{align*}
		where $B = r \cdot \max_{1 \leq j \leq r}B_{j}$.
		By choosing $k = [c \log{T} / \log{\log{T}}]$
		for a suitably small constant $c = c(\s, L, \| \bm{a} \|, r)$, we complete the proof.
	\end{proof}

	Next we compute the moment generating function of $(P_{j}(\s + it, X) )_{j = 1}^{r}$ on $\mathcal A$.

	\begin{proposition} \label{MPJEV}
		Let $\frac{1}{2} < \s < 1$, $L \geq 1$ be fixed.
		There exists a positive constant
		$b_{1} = b_{1}(\s, L, \| \bm{a} \|, r)$ such that
		for large $T$, $X = (\log{T})^{L}$, and $\bm{z} = (z_1, \dots, z_{r}) \in \CC^{r}$
		with $\|\bm{z}\| \leq b_{1}(\log{T})^{\s}$, we have
		\begin{align*}
			& \frac{1}{T}\int_{\mathcal{A}}\exp\l( \sum_{j = 1}^{r} z_{j} \Re P_{j}(\s + it, X) \r)dt \\
			& = \prod_{p \leq X}I_{0}\l(\sqrt{K_{\bm{a}}(p, \bm{z}) / p^{2\s}}\r)
			+ O\l(\exp\l( - b_{1}\frac{\log{T}}{\log{\log{T}}} \r)\r).
		\end{align*}
		Here $I_0(z)$ is the $0$-th modified Bessel function.
	\end{proposition}

	To prove Proposition \ref{MPJEV}, we require two lemmas.

	\begin{lemma} \label{MVERDP}
		For every $1 \leq j \leq r$, define
		\begin{align} \label{def_RDP}
			P_{j}(\s, \mathcal{X}, X) = \sum_{p \leq X}\frac{a_{j}(p) \mathcal{X}(p)}{p^{\s}}.
		\end{align}
		Then there exists a positive constant $C = C(\s, \| a_{j} \|_{\infty})$ such that for any $X \geq 3$, $k \in \ZZ_{\geq 1}$
		\begin{align*}
			\EXP{\l|P_{j}(\s, \mathcal{X}, X)\r|^{k}}
			\leq \l( \frac{C k^{1-\s}}{(\log{2k})^{\s}} \r)^{k}.
		\end{align*}
	\end{lemma}

	\begin{proof}
		Using the Cauchy-Schwarz inequality, we have
		\begin{align*}
			\EXP{\l|P_{j}(\s, \mathcal{X}, X)\r|^{k}}
			\leq \l(\EXP{\l|P_{j}(\s, \mathcal{X}, X)\r|^{2k}}\r)^{1/2}.
		\end{align*}
		Using \eqref{SLL2} instead of \eqref{SLL1} in the proof of \eqref{MVEP}, we can show that
		\begin{align*}
			\EXP{\l|P_{j}(\s, \mathcal{X}, X)\r|^{2k}}
			\leq \l( \frac{C k^{1 - \s}}{(\log{2k})^{\s}} \r)^{2k}
		\end{align*}
		for any $X \geq 3$, $k \in \ZZ_{\geq 1}$ which completes the proof.
	\end{proof}

	\begin{lemma} \label{EQEVI0}
		For any $\bm{z} = (z_{1}, \dots, z_{r}) \in \CC^{r}$, we have
		\begin{align*}
			&\EXP{\exp\l( \sum_{j = 1}^{r}z_{j} \Re \frac{a_{j}(p)}{p^{\s}} \mathcal{X}(p) \r)}
			= I_{0}(\sqrt{K_{\bm{a}}(p, \bm{z}) / p^{2\s}}).
		\end{align*}
	\end{lemma}

	\begin{proof}
		By the Taylor expansion of $\exp(\cdot)$ and the identity $\Re w = \frac{w + \ol{w}}{2}$, we can write
		\begin{align*}
			&\EXP{\exp\l( \sum_{j = 1}^{r}z_{j} \Re \frac{a_{j}(p)}{p^{\s}} \mathcal{X}(p) \r)}\\
			&= \sum_{n = 0}^{\infty}\frac{1}{2^{n} p^{n \s} n!}
			\EXP{ \l( \sum_{j = 1}^{r}z_{j}\l( a_{j}(p) \mathcal{X}(p) + \ol{a_{j}(p) \mathcal{X}(p)} \r) \r)^{n}}.
		\end{align*}
		Using the binomial expansion, we find that
		\begin{align*}
			&\l( \sum_{j = 1}^{r}z_{j}\l( a_{j}(p) \mathcal{X}(p) + \ol{a_{j}(p) \mathcal{X}(p)} \r) \r)^{n}
			= \sum_{\ell = 0}^{n}\binom{n}{\ell}
			\l( \sum_{j = 1}^{r}z_{j}a_{j}(p) \r)^{\ell}
			\l( \sum_{k = 1}^{r}z_{k}\ol{a_{k}(p)} \r)^{n - \ell}
			\mathcal{X}(p)^{2\ell - n}.
		\end{align*}
		Since $\mathcal{X}(p)$ is uniformly distributed on the unit circle in $\CC$, we have
		\begin{align} \label{BEQXp}
			\EXP{\mathcal{X}(p)^{a}}
			= \l\{
			\begin{array}{cc}
				1 & \text{if \; $a = 0$,} \\
				0 & \text{otherwise}
			\end{array}
			\r.
		\end{align}
		for any $a \in \ZZ$.
		Hence, we obtain
		\begin{align*}
			\EXP{\exp\l( \sum_{j = 1}^{r}z_{j} \Re \frac{a_{j}(p)}{p^{\s}} \mathcal{X}(p) \r)}
			&= \sum_{n = 0}^{\infty} \frac{1}{2^{2n} p^{2n\s} (2n)!}\binom{2n}{n}
			K_{\bm{a}}(p, \bm{z})^{n}\\
			&= \sum_{n = 0}^{\infty} \frac{1}{(n!)^{2}}\l( \sqrt{K_{\bm{a}}(p, \bm{z}) / p^{2\s}} / 2 \r)^{2n}
			= I_{0}(\sqrt{K_{\bm{a}}(p, \bm{z}) / p^{2\s}}).
		\end{align*}
	\end{proof}

	\begin{proof}[Proof of Proposition \ref{MPJEV}]
		Let $L \geq 1$ be fixed and $T$, $X$, and $Y$ be large numbers such that $X \leq (\log{T})^{L}$ and $Y = \frac{\log{T}}{4 L \log{\log{T}}}$.
		Let $\bm{z} = (z_{1}, \dots, z_{r}) \in \CC^{r}$ with $\|\bm{z}\| \leq b_{1}(\log{T})^{\s}$,
		where $b_{1} = b_{1}(\s, L, \| \bm{a} \|, r)$ is a suitably small constant to be chosen later.
		From \eqref{def_sA_JEV}, we have
		\begin{align*}
			& \frac{1}{T}\int_{\mathcal{A}}\exp\l( \sum_{j = 1}^{r} z_{j} \Re P_{j}(\s + it, X) \r)dt \\
			& = \frac{1}{T}\sum_{0 \leq k \leq Y}\frac{1}{k!}\int_{\mathcal{A}}
			\bigg(\sum_{j = 1}^{r} z_{j} \Re P_{j}(\s + it, X) \bigg)^{k}dt
			+ O\l(\sum_{k > Y}\frac{1}{k!}\l(r \|\bm{z}\| \frac{(\log{T})^{1 - \s}}{\log{\log{T}}} \r)^{k} \r).
		\end{align*}
		From the choice of $Y$ and the bound for $\|\bm z\|$,
		we see that this $O$-term is bounded by $\ll \exp\l( - \log{T} / \log{\log{T}} \r)$.
		Using the Cauchy-Schwarz inequality, we find that
		\begin{align*}
			& \frac{1}{T}\int_{\mathcal{A}}\bigg(\sum_{j = 1}^{r} z_{j} \Re P_{j}(\s + it, X) \bigg)^{k}dt \\
			& = \frac{1}{T}\int_{T}^{2T}\bigg(\sum_{j = 1}^{r} z_{j} \Re P_{j}(\s + it, X) \bigg)^{k}dt    \\
			&\quad+ O\l( \frac{1}{T}(\meas([T, 2T] \setminus \mathcal{A}))^{1/2}
			\l(\int_{T}^{2T} \bigg|\sum_{j = 1}^{r} z_{j} \Re P_{j}(\s + it, X) \bigg|^{2k}dt\r)^{1/2} \r).
		\end{align*}
		By Lemma \ref{ESAEG_JEV}, estimate \eqref{MVEP}, and the bound for $\|\bm{z}\|$, this $O$-term is
		\begin{align*}
			&\ll \exp\l( -\frac{b_{0}}{2} \frac{\log{T}}{\log{\log{T}}} \r)
			\l( C_{1} b_{1} (\log{T})^{\s} \frac{k^{1 - \s}}{(\log{2k})^{\s}}  \r)^{k}\\
			&\ll \exp\l( -\frac{b_{0}}{2} \frac{\log{T}}{\log{\log{T}}} \r)
			\l( 2 C_{1} b_{1} \frac{\log{T}}{\log{\log{T}}}  \r)^{k}
		\end{align*}
		for $0 \leq k \leq Y$, where $C_{1} = C_{1}(\s, \| \bm{a} \|) > 0$ is a positive constant. It follows that
		\begin{align}	\label{GRKLG5}
			& \frac{1}{T}\int_{\mathcal{A}}\exp\l( \sum_{j = 1}^{r} z_{j} \Re P_{{j}}(\s + it, X) \r)dt \\
			& = \frac{1}{T}\sum_{0 \leq k \leq Y}\frac{1}{k!}
			\int_{T}^{2T}\bigg(\sum_{j = 1}^{r} z_{j} \Re P_{{j}}(\s + it, X) \bigg)^{k}dt               \\
			& \quad + O\l(\exp\l( -\frac{b_{0}}{2} \frac{\log{T}}{\log{\log{T}}} \r)
			\sum_{0 \leq k \leq Y}\frac{1}{k!}\l( 2 C_{1} b_{1} \frac{\log{T}}{\log{\log{T}}}  \r)^{k}\r).
		\end{align}
		This $O$-term is
		\begin{align*}
			\ll  \exp\l( -\frac{b_{0}}{2} \frac{\log{T}}{\log{\log{T}}} \r)
			\sum_{k = 0}^{\infty}\frac{1}{k!}\l( 2 C_{1} b_{1} \frac{\log{T}}{\log{\log{T}}}  \r)^{k}
			& = \exp\l( -\l(\frac{b_{0}}{2} - 2C_{1} b_{1}\r) \frac{\log{T}}{\log{\log{T}}} \r).
		\end{align*}
		Hence the $O$-term on the right hand side of \eqref{GRKLG5} is $\ll \exp\l( -\frac{b_{0}}{4}\frac{\log{T}}{\log{\log{T}}} \r)$ when $b_{1}$ is sufficiently small.

		Next, we write
		\begin{align*}
			& \int_{T}^{2T}\bigg(\sum_{j = 1}^{r} z_{j} \Re P_{{j}}(\s + it, X) \bigg)^{k}dt      \\
			& = \sum_{1 \leq j_{1}, \dots, j_{k} \leq r}z_{j_{1}} \cdots z_{j_{k}}
			\sum_{p_{1}, \dots, p_{k} \leq X}\frac{1}{(p_{1} \cdots p_{k})^{\s}}
			\int_{T}^{2T}\l(\Re a_{j_{1}}(p_{1}) p_{1}^{-i t} \r)
			\cdots \l(\Re a_{j_{k}}(p_{k}) p_{k}^{-i t} \r)dt.
		\end{align*}
		From this equation and Lemma \ref{GRLR}, we have
		\begin{align*}
			& \frac{1}{T}\int_{T}^{2T}\bigg(\sum_{j = 1}^{r} z_{j} \Re P_{{j}}(\s+it, X) \bigg)^{k}dt \\
			&= \EXP{\bigg(\sum_{j = 1}^{r} z_{j} \Re P_{{j}}(\s, \mathcal{X}, X) \bigg)^{k}}         \\
			&\quad+O\l( \frac{1}{T}\sum_{1 \leq j_{1}, \dots, j_{k} \leq r}|z_{j_{1}} \cdots z_{j_{k}}|
			\sum_{p_{1}, \dots, p_{k} \leq X}|a_{j_{1}}(p_{1}) \cdots a_{j_{k}}(p_{k})|
			(p_{1} \cdots p_{k})^{1 - \s} \r),
		\end{align*}
		where $P_{j}(\s, \mathcal{X}, X)$ is defined by \eqref{def_RDP}.
		Additionally, we see that this $O$-term is
		\begin{align*}
			\ll \frac{1}{T}\l( \sum_{j = 1}^{r}|z_{j}|\sum_{p \leq X}|a_{j}(p)|p^{1 - \s} \r)^{k}
			\leq \frac{(C \| \bm{z} \| X^2)^{k}}{T}
			\leq \frac{C^{2k}}{T^{1/3}}
			\leq \exp\l( -\frac{\log{T}}{\log{\log{T}}} \r)
		\end{align*}
		for $0 \leq k \leq Y$ when $T$ is sufficiently large.
		Here, $C$ is a positive constant may depend on $\| \bm{a} \|$.
		Therefore, we have
		\begin{align*}
			& \frac{1}{T}\int_{\mathcal{A}}\exp\l( \sum_{j = 1}^{r} z_{j} \Re P_{j}(\s+it, X) \r)dt                     \\
			& = \sum_{0 \leq k \leq Y}\frac{1}{k!}\EXP{\l( \sum_{j = 1}^{r}z_{j}\Re P_{j}(\s, \mathcal{X}, X) \r)^{k}}
			+ O\l( \exp\l( -\frac{b_{0}}{4}\frac{\log{T}}{\log{\log{T}}} \r) \r)                                                                      \\
			& = \EXP{\exp\l( \sum_{j = 1}^{r}z_{j} \Re P_{{j}}(\s, \mathcal{X}, X) \r)}
			- \sum_{k > Y}\frac{1}{k!}\EXP{\l( \sum_{j = 1}^{r}z_{j}\Re P_{{j}}(\s, \mathcal{X}, X) \r)^{k}}              \\
			& \quad + O\l( \exp\l( -\frac{b_{0}}{4}\frac{\log{T}}{\log{\log{T}}} \r) \r).
		\end{align*}
		Using Lemma \ref{MVERDP} and the bound for $\| \bm{z} \|$, we obtain
		\begin{align}
			\EXP{\l| \sum_{j = 1}^{r}z_{j}\Re P_{{j}}(\s, \mathcal{X}, X) \r|^{k}}
			\leq b_{1}^{k} (\log{T})^{\s k} \l(\frac{C k^{1-\s}}{(\log{2k})^{\s}}\r)^{k}
		\end{align}
		for some constant $C = C(\bm{a}, \s) > 0$.
		Hence it holds that
		\begin{align*}
			\l|\sum_{k > Y}\frac{1}{k!}\EXP{\l( \sum_{j = 1}^{r}z_{j}\Re P_{{j}}(\s, \mathcal{X}, X) \r)^{k}}\r|
			&\leq \sum_{k > Y}\frac{\l(C b_{1} (\log{T})^{\s}\r)^{k}}{(k \log{k})^{\s k}}
			\leq \sum_{k > Y}e^{-k}\\
			&\ll \exp\l( - \frac{\log{T}}{4L \log{\log{T}}} \r).
		\end{align*}
		We find from these that for any $\|\bm{z}\| \leq b_{1}(\log{T})^{\s}$ with $b = b_{1}(\s, L, \| \bm{a} \|, r) > 0$ sufficiently small
		\begin{align} \label{GRKLG6}
			&\frac{1}{T}\int_{\mathcal{A}}\exp\l( \sum_{j = 1}^{r} z_{j} \Re P_{j}(\s+it, X) \r)dt\\
			&= \EXP{\exp\l( \sum_{j = 1}^{r}z_{j} \Re P_{j}(\s, \mathcal{X}, X) \r)}
			+ O\l( \exp\l(- b \frac{\log{T}}{\log{\log{T}}} \r)\r)\\
			&= \prod_{p \leq X}\EXP{\exp\l( \sum_{j = 1}^{r}z_{j} \Re \frac{a_{j}(p)}{p^{\s}}\mathcal{X}(p) \r)}
			+ O\l( \exp\l(- b \frac{\log{T}}{\log{\log{T}}} \r)\r)
		\end{align}
		by using the independence of $\mathcal{X}(p)$.
		By this equation and Lemma \ref{EQEVI0}, we complete the proof of Proposition \ref{MPJEV}.
	\end{proof}

\section{\textbf{Distribution of Dirichlet polynomials in the strip $\frac{1}{2} < \s < 1$}}

	Let $\bm \chi=(\chi_1, \dots, \chi_r)$ be an $r$-tuple of Dirichlet characters, and $\bm{\theta} \in \RR^{r}$.
	For $\bm{x} = (x_{1}, \dots, x_{r}) \in \RR^{r}$ define
	\begin{gather}
		\label{def_K_chi}
		K_{\bm{\chi}, \bm{\theta}}(n, \bm{x})
		\ceq \bigg|\sum_{j=1}^{r}x_{j} e^{-i\theta_{j}}\chi_{j}(n)\bigg|^{2},\\
		\label{def_F_chi}
		F_{\bm{\chi}, \bm{\theta}, \s}(\bm{x})
		\ceq \frac{1}{\phi(d)}\sum_{u \in (\ZZ / d \ZZ)^{\times}}K_{\bm{\chi}, \bm{\theta}}(u, \bm{x})^{\frac{1}{2\s}}.
	\end{gather}

	We evaluate the main term in Proposition \ref{MPJEV} for $K_{\bm{\chi}, \bm{\theta}}(p, \bm{x})$ in this section.

	\begin{proposition}	\label{AFIQ}
		Under the notation above, for any $X \geq 30$, $3 \leq x_{j} \leq X^{\frac{2\s}{3}}$, we have
		\begin{align*}
			\prod_{p \leq X}I_{0}\l( \sqrt{K_{\bm{\chi}, \bm{\theta}}(p, \bm{x}) / p^{2\s}} \r)
			= \exp\l(\frac{G(\s)}{\log{\| \bm{x} \|}}
			\l( F_{\bm{\chi}, \bm{\theta}, \s}(\bm{x})
			+ O_{\s, \bm{\chi}}\l( \frac{\| \bm{x} \|^{\frac{1}{\s}}\log{\log{\| \bm{x} \|}}}{\log{\| \bm{x} \|}} \r) \r) \r),
		\end{align*}
		where $G(\s)$ is defined in \eqref{def_G_s}.
	\end{proposition}

	\begin{lemma}	\label{KLAFIQ}
		Let $d, u \in \mathbb{Z}\setminus\{0\}$ with $(d, u)=1$.
		Let $ \frac{1}{2} < \s < 1$ be fixed.
		For $X \geq 3$ and $0 \leq x \leq X^{\frac{2\s}{3}}$, we have
		\begin{align*}
			\sum_{\substack{p \leq X \\ p\equiv u \pmod d }}\log{I_{0}\l( \frac{x}{p^{\s}} \r)}
			= \frac{G(\s) x^{\frac{1}{\s}}}{\phi(d)\log{(x + 2)}}\l( 1 + O_{d, \sigma}\l( \frac{1}{\log{(x + 2)}} \r) \r).
		\end{align*}
	\end{lemma}

	\begin{proof}
		Using the Taylor expansion of $I_{0}$ we find that for $0 \leq x \leq 2$
		\begin{align*}
			\sum_{\substack{p \leq X \\ p\equiv u \pmod d }}\log{I_{0}\l( \frac{x}{p^{\s}} \r)}
			\ll \sum_{p \leq X} \frac{x^2}{p^{2\s}}
			\ll_{\s} x^2
			\ll \frac{x^{\frac{1}{\s}}}{(\log(x + 2))^2},
		\end{align*}
		which proves the lemma when $0 \leq x \leq 2$.
		From now on, we assume that $x \geq 2$. Let $y_0, y_1$ be some parameters to be chosen later.
		We split the range of $p\leq X$ into three ranges
		\begin{align*}
			\sum_{\substack{p \leq X \\ p\equiv u \pmod d }}\log{I_{0}\l( \frac{x}{p^{\s}} \r)}
			& = \l(\sum_{\substack{p \leq y_{0}           \\ p\equiv u\pmod d}} + \sum_{\substack{y_{1} < p \leq X \\ p\equiv u\pmod d}}
			+ \sum_{\substack{y_{0} < p \leq y_{1}         \\ p\equiv u\pmod d}}\r)\log{I_{0}\l( \frac{x}{p^{\s}} \r)}
			\eqc S_{1}+ S_{2} + S_{3}.
		\end{align*}
		From the prime number theorem in arithmetic progressions we have
		\begin{align}\label{BFQC}
			A(y;u,d)
			\ceq \sum_{\substack{p \leq y \\ p\equiv u\pmod d}}1
			= \frac{\li(y)}{\phi(d)} + O_{d}\l( y\exp\l(-c \sqrt{\log{y}}\r) \r)
		\end{align}
		where $\li(y) \ceq \int_{2}^{y}\frac{du}{\log{u}}$.
		By partial summation, we find that
		\begin{align}
			\label{EqSpm1}
			S_{3}
			= & -\int_{y_{0}}^{y_{1}}A(\xi;u,d)\l( \frac{d}{d\xi}\log{I_{0}\l( \frac{x}{\xi^{\s}} \r)} \r)d\xi
			+ A(y_{1};u,d)\log{I_{0}\l( \frac{x}{y_{1}^{\s}} \r)}                                              \\
			& -A(y_{0};u,d)\log{I_{0}\l( \frac{x}{y_{0}^{\s}} \r)}.
		\end{align}
		By  equation \eqref{BFQC}, the integral on the right hand side is equal to
		\begin{align*}
			-\frac{1}{\phi(d)}\int_{y_{0}}^{y_{1}}\li(\xi)\l( \frac{d}{d\xi}\log{I_{0}\l( \frac{x}{\xi^{\s}} \r)} \r)d\xi
			+ O_{d}\l( \int_{y_{0}}^{y_{1}}\xi e^{-c\sqrt{\log{\xi}}}
			\l( \frac{d}{d\xi}\log{I_{0}\l( \frac{x}{\xi^{\s}} \r)} \r)d\xi \r).
		\end{align*}
		Note that we used the monotonicity of $I_{0}$ in the above deformation.
		We also have
		\begin{align*}
			& -\int_{y_{0}}^{y_{1}}\li(\xi)\l( \frac{d}{d\xi}\log{I_{0}\l( \frac{x}{\xi^{\s}} \r)} \r)d\xi \\
			& = -\li(y_{1})\log{I_{0}\l( \frac{x}{y_{1}^{\s}} \r)}
			+ \li(y_{0})\log{I_{0}\l( \frac{x}{y_{0}^{\s}} \r)}
			+ \int_{y_{0}}^{y_{1}}\frac{\log{I_{0}\l( \frac{x}{\xi^{\s}} \r)}}{\log{\xi}}d\xi,
		\end{align*}
		and
		\begin{align*}
			& \int_{y_{0}}^{y_{1}}\xi e^{-c\sqrt{\log{\xi}}}
			\l( \frac{d}{d\xi}\log{I_{0}\l( \frac{x}{\xi^{\s}} \r)} \r)d\xi                                            \\
			& \ll y_{1}e^{-c\sqrt{\log{y_{1}}}}\log{I_{0}\l( \frac{x}{y_{1}^{\s}} \r)}
			+ y_{0}e^{-c\sqrt{\log{y_{0}}}}\log{I_{0}\l( \frac{x}{y_{0}^{\s}} \r)}
			+ \int_{y_{0}}^{y_{1}}e^{-c\sqrt{\log{\xi}}}\log{I_{0}\l( \frac{x}{\xi^{\s}} \r)}d\xi \\
			& \ll x^{2}y_{1}^{1-2\s}e^{-c\sqrt{\log{y_{1}}}} + x y_{0}^{1-\s}e^{-c\sqrt{\log{y_{0}}}}.
		\end{align*}
		Substituting the above estimates to \eqref{EqSpm1} and using \eqref{BFQC}, we obtain
		\begin{align*}
			S_{3}
			= & \frac{1}{\phi(d)}\int_{y_{0}}^{y_{1}}\frac{\log{I_{0}\l( \frac{x}{\xi^{\s}} \r)}}{\log{\xi}}d\xi
			+ O_{d}\l(x^{2}y_{1}^{1-2\s}e^{-c\sqrt{\log{y_{1}}}} + x y_{0}^{1-\s}e^{-c\sqrt{\log{y_{0}}}}\r).
		\end{align*}
		By making the change of variables $u = \frac{x}{\xi^{\s}}$, we have
		\begin{align*}
			\int_{y_{0}}^{y_{1}}\frac{\log{I_{0}\l( \frac{x}{\xi^{\s}} \r)}}{\log{\xi}}d\xi
			= x^{\frac{1}{\s}}\int_{x/y_{1}^{\s}}^{x/y_{0}^{\s}}
			\frac{\log{I_{0}(u)}}{u^{1 + \frac{1}{\s}}\log{(x / u)}}du.
		\end{align*}
		We have that
		$
			\frac{1}{\log(x / u)}
			= \frac{1}{\log{x}} + O\l( \frac{|\log{u}|}{(\log{x})^{2}} \r)
		$
		for $x^{-1/2} \leq u \leq x^{1/2}$.
		Therefore, by choosing $y_{0} = x^{\frac{1}{2\s}}$, $y_{1} = x^{\frac{3}{2\s}}$ we find that the right hand side above is equal to
		\begin{align}	\label{KLAFIQ1}
			\frac{x^{\frac{1}{\s}}}{\log{x}}\int_{x/y_{1}^{\s}}^{x/y_{0}^{\s}}
			\frac{\log{I_{0}(u)}}{u^{1 + \frac{1}{\s}}}du
			+ O\l( \frac{x^{\frac{1}{\s}}}{(\log{x})^{2}}\int_{x/y_{1}^{\s}}^{x/y_{0}^{\s}}
			\frac{\log{I_{0}(u)} |\log{u}|}{u^{1 + \frac{1}{\s}}}du \r).
		\end{align}
		Moreover, we find that the main term of \eqref{KLAFIQ1} is equal to
		\begin{align*}
			& \frac{x^{\frac{1}{\s}}}{\log{x}}\int_{0}^{\infty}\frac{\log{I_{0}(u)}}{u^{1 + \frac{1}{\s}}}du
			+ O_{\s}\l( \frac{x^{1/\s}}{\log{x}}\l( \Big(\frac{x}{ y_{1}^\sigma}\Big)^{\frac{2\s-1}{\s}} + \Big(\frac{x}{ y_{0}^\sigma}\Big)^{\frac{\s - 1}{\s}} \r) \r)
			= \frac{G(\s)x^{\frac{1}{\s}}}{\log{x}} + O_{\s}\l( \frac{x^{\frac{1}{\s}}}{(\log{x})^{2}} \r),
		\end{align*}
		and that the $O$-term of \eqref{KLAFIQ1} is
		\begin{align*}
			\ll \frac{x^{\frac{1}{\s}}}{(\log{x})^{2}}\int_{0}^{\infty}
			\frac{\log{I_{0}(u)} |\log{u}|}{u^{1 + \frac{1}{\s}}}du
			\ll_{\s} \frac{x^{\frac{1}{\s}}}{(\log{x})^{2}}.
		\end{align*}
		Hence, we have
		\begin{align*}
			S_{3}
			= \frac{G(\s) x^{\frac{1}{\s}}}{\phi(d)\log{x}}\l( 1 + O_{d, \s}\l( \frac{1}{\log{x}} \r) \r).
		\end{align*}
		We also find that
		\begin{align*}
			S_{1}
			\leq \sum_{p \leq y_{0}}\frac{x}{p^{\s}}
			\ll x^{\frac{1+\sigma}{2\sigma}} \ll_{\s} \frac{x^{\frac{1}{\s}}}{(\log{x})^2}
		\end{align*}
		by the inequality $I_{0}(x / p^{\s}) \leq \exp(x / p^{\s})$,
		and that
		\begin{align*}
			S_{2}
			\ll \sum_{p > y_{1}}\frac{x^2}{p^{2\s}}
			\ll_{\s} \frac{x^{2}}{y_{1}^{2\s - 1}}
			\ll_{\s} \frac{x^{\frac{1}{\s}}}{(\log{x})^2}
		\end{align*}
		by using the Taylor expansion of $I_{0}$.
		Thus, we obtain Lemma \ref{KLAFIQ}.
	\end{proof}

	\begin{proof}[Proof of Proposition \ref{AFIQ}]
		Let $\bm x=(x_1, \dots, x_r) \in (\RR_{\geq 3})^r$, and $\bm{\theta} \in \RR^{r}$.
		When $\chi_j$ is a Dirichlet character modulo $q_{j}$,
		we take $d=\operatorname{lcm}(q_1, \dots, q_{r})$ and split the sum over $p$ into residue classes modulo $d$ to obtain
		\begin{align}
			&\sum_{p\leq X}\log I_{0}\l(\sqrt{K_{\bm{\chi}, \bm{\theta}}(p, \bm x)/p^{2\sigma}}\r)\\
			&= \sum_{u \in (\ZZ / d \ZZ)^{\times}} \sum_{\substack{p \leq X \\ p\equiv u\pmod d}}
			\log I_{0}\l(\sqrt{K_{\bm{\chi}, \bm{\theta}}(u, \bm x)/p^{2\sigma}}\r)
			+ \sum_{\substack{p \leq X \\ p \mid d}}\log{I_{0}\l(\sqrt{K_{\bm{\chi}, \bm{\theta}}(p, \bm x)/p^{2\sigma}}\r)}.
		\end{align}
		Since $|\log{I_{0}(x)}| \leq x$ and $|K_{\bm{\chi}, \bm{\theta}}(p, \bm{x})| \leq \| \bm{x} \|^2$ it holds that
		\begin{align}
			\sum_{p \mid d} \log{I_{0}\l( \sqrt{K_{\bm{\chi}, \bm{\theta}}(u, \bm{x})/p^{2\sigma}} \r)}
			\leq \sum_{p \mid d} \frac{1}{p^{\s}}\| \bm{x} \|
			\ll_{\bm{\chi}} \| \bm{x} \|.
		\end{align}
		By Lemma \ref{KLAFIQ}, we have
		\begin{align}
			&\sum_{\substack{p \leq X \\ p\equiv u\pmod d}}
			\log I_{0}\l(\sqrt{K_{\bm{\chi}, \bm{\theta}}(u, \bm x)/p^{2\sigma}}\r)\\
			&= \frac{G(\s) K_{\bm{\chi}, \bm{\theta}}(u, \bm{x})^{\frac{1}{2\s}}}{\phi(d) \log(K_{\bm{\chi}, \bm{\theta}}(u, \bm{x})^{1/2} + 2)}
			\l( 1 + O_{\s, \bm{\chi}}\l( \frac{1}{\log(K_{\bm{\chi}, \bm{\theta}}(u, \bm{x})^{1/2} + 2)} \r) \r).\label{Gsigma}
		\end{align}
		If $K_{\bm{\chi}, \bm{\theta}}(u, \bm{x}) \geq \frac{\| \bm{x} \|^2}{(\log{\| \bm{x} \|})^{2}}$, then \eqref{Gsigma} becomes \begin{align*}
			\frac{G(\s)}{\phi(d) \log{\| \bm{x} \|}}
			\l( K_{\bm{\chi}, \bm{\theta}}(u, \bm{x})^{\frac{1}{2\s}}
			+ O_{\s, \bm{\chi}}\l( \frac{\| \bm{x} \|^{\frac{1}{\s}}\log{\log{\| \bm{x} \|}}}{\log{\| \bm{x} \|}} \r) \r)
		\end{align*} by using the trivial bound $K_{\bm{\chi}, \bm{\theta}}(u, \bm{x}) \leq \| \bm{x} \|^{2}$.
		If $K_{\bm{\chi}, \bm{\theta}}(u, \bm{x}) \leq \frac{\| \bm{x} \|^2}{(\log{\| \bm{x} \|})^{2}}$, then we have \eqref{Gsigma} is
		$
			\ll_{\s, \bm{\chi}} \frac{\| \bm{x} \|^{\frac{1}{\s}}}{(\log{\| \bm{x} \|})^2}.
		$
		Hence, we obtain
		\begin{align*}
			&\sum_{p\leq X}\log I_{0}\l(\sqrt{K_{\bm{\chi}, \bm{\theta}}(p, \bm x)/p^{2\sigma}}\r)\\
			&= \frac{G(\s)}{\phi(d) \log{\|\bm{x}\|}}\sum_{u \in (\ZZ / d \ZZ)^{\times}}
			\l( K_{\bm{\chi}, \bm{\theta}}(u, \bm{x})^{\frac{1}{2\s}}
			+ O_{\s, \bm{\chi}}\l( \frac{\| \bm{x} \|^{\frac{1}{\s}}\log{\log{\| \bm{x} \|}}}{\log{\| \bm{x} \|}} \r) \r)
			+ O_{\bm{\chi}}\l( \| \bm{x} \| \r)\\
			&= \frac{G(\s)}{\log{\|\bm{x}\|}}
			\l( F_{\bm{\chi}, \bm{\theta}, \s}(\bm{x})
			+ O_{\s, \bm{\chi}}\l( \frac{\| \bm{x} \|^{\frac{1}{\s}}\log{\log{\| \bm{x} \|}}}{\log{\| \bm{x} \|}} \r) \r),
		\end{align*}
		which completes the proof of Proposition \ref{AFIQ}.
	\end{proof}

	Now, we are ready to prove Proposition \ref{Main_Prop_JEV}.

	\begin{proof}[Proof of Proposition \ref{Main_Prop_JEV}]
		Let $\bm{\alpha} \in (\RR_{> 0})^{r}$ such that
		$\Xi_{1}(\s, \bm{\chi}, \bm{\theta}; \bm{\a}), \dots, \Xi_{r}(\s, \bm{\chi}, \bm{\theta}; \bm{\a}) > 0$.
		Let $L \geq 2$.
		Let $T$ be large, and $X = (\log{T})^{L}$.
		Let $V$ be large with $V \leq \frac{a_{2} (\log{T})^{1-\s}}{\log{\log{T}}}$,
		where $a_{2} = a_{2}(\s, L, \bm{\a})$ is a suitably positive constant to be chosen later.
		Here, we let the positive parameters $x_{1}, \dots, x_{r}$ to be the solutions of the equations
		\begin{align}	\label{def_xs}
			V
			&= \frac{G(\s)}{\s \log{(x_{1} / \a_{1})}} (x_{1} / \a_{1})^{\frac{1}{\s} - 1}
			= \cdots
			= \frac{G(\s)}{\s \log{(x_{r} / \a_{r})}} (x_{r} / \a_{r})^{\frac{1}{\s} - 1}.
		\end{align}
		When $V$ is sufficiently large, we can verify that $x_{i} / \a_{i} = x_{j} / \a_{j}$ for all $i, j = 1, \dots, r$,
		and that
		\begin{gather}
			\label{DPJEV4}
			\frac{x_{j}}{\a_{j}}
			= \frac{A(\s)}{1 - \s}\l( V \log{V} \r)^{\frac{\s}{1-\s}}
			\l( 1 + O_{\s, \bm{\chi}, \bm{\a}}\l( \frac{\log{\log{V}}}{\log{V}} \r) \r).
		\end{gather}
		To use Propositions \ref{MPJEV}, \ref{AFIQ}, we choose $a_{2} = a_{2}(\s, L, \bm{\a})$
		such that $\| \bm{x} \| \leq \frac{b_{1}}{2}(\log{T})^{\s}$,
		where $b_{1} = b_{1}(\s, L, r)$ is the same constant as in Proposition \ref{MPJEV}
		with $a_{j}(p)=\chi_{j}(p)e^{-i\theta_{j}}$.

		For any Lebesgue measurable set $S \subset [T, 2T]$ and $\bm{v} = (v_{1}, \dots, v_{r}) \in \RR^{r}$, denote
		\begin{align} \label{def_Psi_S}
			&\tilde\Psi_{S}(T, \bm{v}, X)\\
			&\ceq \frac{1}{T}\meas \set{t \in S}{\Re\sum_{p \leq X}\frac{\chi_j(p)e^{-i\theta_j}}
			{p^{\s + it}} > \Xi_{j}(\s, \bm{\chi}, \bm{\theta}; \bm{\a})v_{j} \text{\; for all $j=1, \dots, r$}}.
		\end{align}
		Let $y_{j} = x_{j} \Xi_{j}(\s, \bm{\chi}; \bm{\a})$. Then we find that
		\begin{align}	\label{KEJEV}
			&y_{1} \cdots y_{r} \int_{-\infty}^{\infty} \cdots \int_{-\infty}^{\infty}e^{y_{1} v_{1} + \cdots + y_{r} v_{r}}
			\tilde\Psi_{\mathcal{A}}(T, (v_{1}, \dots, v_{r}), X)d v_{1} \cdots d v_{r}\\
			& =\frac{1}{T} \int_{\mathcal{A}}
			\exp\l( \sum_{j = 1}^{r} x_{j}\Re\sum_{p \leq X}\frac{\chi_{j}(p)e^{-i\theta_j}}{p^{\s + it}}\r)dt,
		\end{align}
		where $\mathcal{A} = \mathcal{A}(T, X, \bm{\chi}, \bm{\theta})$ is the set defined by \eqref{def_sA_JEV} with $a_{j}(p)=\chi_{j}(p)e^{-i\theta_{j}}$.
		By this equation and Propositions \ref{MPJEV}, \ref{AFIQ}, we have
		\begin{align}	\label{DPJEV1}
			&\int_{-\infty}^{\infty} \cdots \int_{-\infty}^{\infty}e^{y_{1}v_{1} + \cdots + y_{r}v_{r}}
			\tilde\Psi_{\mathcal{A}}(T, (v_{1}, \dots, v_{r}), X) d v_{1} \cdots d v_{r}\\
			& = \exp\l(\frac{G(\s)}{\log{\| \bm{x} \|}}
			\l( F_{\bm{\chi}, \bm{\theta}, \s}(\bm{x})
			+ O_{\s, \bm{\chi}}\l( \frac{\| \bm{x} \|^{\frac{1}{\s}}\log{\log{\| \bm{x} \|}}}{\log{\| \bm{x} \|}} \r) \r) \r).
		\end{align}
		Next, we divide the range of the integral of \eqref{DPJEV1} as follows:
		\begin{align}
			\label{DPJEV2}
			\begin{aligned}
				\int_{-\infty}^{\infty} \cdots \int_{-\infty}^{\infty}
				= & \int_{V(1 - \delta)}^{V(1 + \delta)} \cdots \int_{V(1 - \delta)}^{V(1 + \delta)}
				+ \sum_{k = 1}^{r} \l(\int \cdots \int_{D_{k}^{+}} + \int \cdots \int_{D_{k}^{-}}\r),
			\end{aligned}
		\end{align}
		where
		\begin{align*}
			\int \cdots \int_{D_{k}^{+}}
			= \int_{V(1 - \delta)}^{V(1 + \delta)} \cdots \int_{V(1 - \delta)}^{V(1 + \delta)}
			\int_{V(1 + \delta)}^{\infty}\overbrace{\int_{-\infty}^{\infty} \cdots \int_{-\infty}^{\infty}}^{k-1},\\
			\int \cdots \int_{D_{k}^{-}}
			= \int_{V(1 - \delta)}^{V(1 + \delta)} \cdots \int_{V(1 - \delta)}^{V(1 + \delta)}
			\int_{-\infty}^{V(1 - \delta)}\overbrace{\int_{-\infty}^{\infty} \cdots \int_{-\infty}^{\infty}}^{k-1},
		\end{align*}
		for some  small $\delta$ to be chosen later.
		By equation \eqref{DPJEV1}, we find that for any $\epsilon>0$
		\begin{align}
			\label{DPJEVv2_1}
			&\int \cdots \int_{D_{k}^{\pm}} e^{y_{1}v_{1} + \cdots + y_{r}v_{r}}
			\tilde\Psi_{\mathcal{A}}(T, (v_{1}, \dots, v_{r}), X) d v_1 \cdots d v_{r}\\
			&\leq e^{\mp \e y_{k} V(1 \pm \delta)}
			\int_{-\infty}^{\infty} \cdots \int_{-\infty}^{\infty}
			e^{y_{1} v_{1} + \cdots + (1 \pm \e)y_{k} v_{k} + \cdots + y_{r} v_{r}}
			\tilde\Psi_{\mathcal{A}}(T, (v_{1}, \dots, v_{r}), X) d v_{1} \cdots d v_{r}\\
			&= e^{\mp \e y_{k} V(1 \pm \delta)}
			\exp\l(\frac{G(\s)}{\log{\| \bm{x} \|}}
			\l( F_{\bm{\chi}, \bm{\theta}, \s}(x_{1}, \dots, (1 \pm \e)x_{k}, \dots, x_{r})
			+ O_{\s, \bm{\chi}}\l( \frac{\| \bm{x} \|^{\frac{1}{\s}} \log{\log{\| \bm{x} \|}}}{\log{\| \bm{x} \|}} \r) \r) \r).
		\end{align}
		It follows from the mean value theorem that there exist $c_{\pm}\in (0,1)$ such that
		\begin{align}\label{Ftaylor}
			&F_{\bm{\chi}, \bm{\theta}, \s}(x_{1}, \dots, (1 \pm \e)x_{k}, \dots, x_{r})\\
			&= F_{\bm{\chi}, \bm{\theta}, \s}(\bm{x})
			\pm \frac{\partial F_{\bm{\chi}, \bm{\theta}, \s}}{\partial x_{k}}(x_{1}, \dots, (1 \pm c_{\pm} \e)x_{k}, \dots, x_{r}) \e x_{k}.
		\end{align}
		From the definition of $F_{\bm{\chi}, \bm{\theta}, \s}$, we find that
		\begin{align*}
			&\frac{\partial F_{\bm{\chi}, \bm{\theta}, \s}}{\partial x_{k}}(x_{1}, \dots, (1 \pm c_{\pm} \e)x_{k}, \dots, x_{r})\\
				&= \frac{1}{\phi(d)}\sum_{u \in (\ZZ / d \ZZ)^{\times}}\frac{1}{2\s} K_{\bm{\chi}, \bm{\theta}}(u, x_{1}, \dots, (1 \pm c_{\pm} \e)x_{k}, \dots, x_{r})^{\frac{1}{2\s} - 1}
				\frac{\partial K_{\bm{\chi}, \bm{\theta}}}{\partial x_{k}}(u, x_{1}, \dots, (1 \pm c_{\pm} \e)x_{k}, \dots, x_{r})\\
				&= \frac{1}{\s \phi(d)}\sum_{u \in (\ZZ / d\ZZ)^{\times}}
				\l( \frac{x_{k}^{2}}{\a_{k}^{2}} \bigg| \sum_{j = 1}^{r}\a_{j} e^{-i\theta_{j}}\chi_{j}(u) \pm \alpha_k c_{\pm} \e e^{-i\theta_{k}} \chi_{k}(u) \bigg|^2 \r)^{\frac{1}{2\s} - 1}\\
				&\qqqquad\qqqquad\times \frac{x_{k}}{\a_{k}} \l( \Re \overline{e^{-i\theta_{k}}\chi_{k}(u)}
				\sum_{\ell = 1}^{r}\a_{\ell}e^{-i\theta_{\ell}}\chi_{\ell}(u) \pm \alpha_k c_{\pm} \e \r)\\
				&= \frac{(x_{k} / \a_{k})^{\frac{1}{\s} - 1}}{\s \phi(d)}\sums_{u \in (\ZZ / d\ZZ)^{\times}}
				\bigg| \sum_{j = 1}^{r}\a_{j} e^{-i\theta_{j}}\chi_{j}(u)\bigg|^{\frac{1}{\s} - 2} \times \l(1 + O_{\s, \bm{\chi}, \bm{\a}}\l(\e \r)\r)^{\frac{1}{\s} - 2}\\
				&\qqqquad\times \Re \overline{e^{-i\theta_{k}}\chi_{k}(u)}
				\sum_{\ell = 1}^{r}\a_{\ell}e^{-i\theta_{\ell}}\chi_{\ell}(u)
				\l(1 + O_{\s, \bm{\chi}, \bm{\a}}\l(\e \r) \r)
				+ O_{\s, \bm{\chi}, \bm{\a}}\l((\e x_{k})^{\frac{1}{\s} - 1}\r)\\
				&= \s^{-1} \Xi_{k}(\s, \bm{\chi}, \bm{\theta}; \bm{\a}) (x_{k} / \a_{k})^{\frac{1}{\s} - 1}
				+ O_{\s, \bm{\chi}, \bm{\a}}\l( (\e \| \bm{x} \|)^{\frac{1}{\s} - 1} \r)
		\end{align*}
		where $\Xi(\sigma, \bm \chi, \bm \theta; \bm \alpha)$ is defined in \eqref{def_Xi_j}.
		Applying this in \eqref{Ftaylor} and using the definition of $\bm{x}$, we find that \eqref{DPJEVv2_1} is
		\begin{align} \label{nDPJEV1}
			&\leq \exp\l(\frac{G(\s)}{\log{\| \bm{x} \|}}
			\l( F_{\bm{\chi}, \bm{\theta}, \s}(\bm{x})
			- \frac{\e \delta}{2\s} \Xi_{k}(\s, \bm{\chi}, \bm{\theta}; \bm{\a}) x_{k}^{\frac{1}{\s}} \a_{k}^{1 - \frac{1}{\s}} \r) \r)
		\end{align}
		by taking $\delta = K\epsilon^{\frac{1}{\sigma}-1}$ with $\e = (\log{\log{\| \bm{x} \|}} / \log{\| \bm{x} \|})^{\s}$ and $K$ sufficiently large depending on $\sigma, \bm \chi, \bm \alpha$.
		From the assumption that $\Xi_{k}(\s, \bm \chi, \bm \theta;\bm \alpha) > 0$ we can
		choose $K$ large enough depending only on $\s$, $\bm{\chi}, \bm \theta$, and $\bm{\a}$ so that
		\begin{align}\label{Vdelta}
			&\int_{V(1 - \delta)}^{V(1 + \delta)}
			\cdots \int_{V(1 - \delta)}^{V(1 + \delta)}
			e^{y_{1}v_{1} + \cdots + y_{r}v_{r}}\tilde\Psi_{\mathcal{A}}(T, (v_1, \dots, v_{r}), X)d v_1 \cdots d v_r\\
			& = \exp\l(\frac{G(\s)}{\log{\| \bm{x} \|}}
			\l( F_{\bm{\chi}, \bm{\theta}, \s}(\bm{x})
			+ O_{\s, \bm{\chi}, \bm \alpha}\l( \frac{\| \bm{x} \|^{\frac{1}{\s}}\log{\log{\| \bm{x} \|}}}{\log{\| \bm{x} \|}} \r) \r) \r).
		\end{align}
		by equation \eqref{DPJEV1} and \eqref{DPJEV2}.

		Note that
		\begin{align*}
			&\int_{V(1 - \delta)}^{V(1 + \delta)}
			\cdots \int_{V(1 - \delta)}^{V(1 + \delta)}
			e^{y_{1}v_{1} + \cdots + y_{r}v_{r}} dv_{1} \cdots dv_{r}
			= \exp((y_{1} + \cdots + y_{r}) V (1 + O(\delta))),
		\end{align*}
		which, together with $\eqref{Vdelta}$, gives
		\begin{align}
			& \tilde\Psi_{\mathcal{A}}(T, (V(1 + \delta), \dots, V(1 + \delta)), X)\\
			\label{nDPJEV3}
			& \leq \exp\l(
			\frac{G(\s)}{\log{\| \bm{x} \|}}
			\l( F_{\bm{\chi}, \bm{\theta}, \s}(\bm{x})
			+ O_{\s, \bm{\chi}, \bm \alpha}\l( \frac{\| \bm{x} \|^{\frac{1}{\s}} \log{\log{\| \bm{x} \|}}}{\log{\| \bm{x} \|}} \r) \r)
			- (y_{1} + \cdots + y_{r}) V (1 + O(\delta)) \r)\\
			& \leq \tilde\Psi_{\mathcal{A}}(T, (V(1 - \delta), \dots, V(1 - \delta)), X).
		\end{align}
		It follows by \eqref{def_xi_chi}, \eqref{def_Xi_j}, and $\frac{x_{1}}{\a_{1}} = \frac{x_{2}}{\a_{2}} = \cdots = \frac{x_{r}}{\a_{r}}$ that
		\begin{align*}
			(y_{1} + \cdots + y_{r}) V
			&= \l( \a_{1} \Xi_{1}(\s, \bm{\chi}, \bm{\theta}; \bm{\a})\frac{x_{1}}{\a_{1}}
			+ \cdots + \a_{r} \Xi_{r}(\s, \bm{\chi}, \bm{\theta}; \bm{\a}) \frac{x_{r}}{\a_{r}} \r) V
			= \xi(\s, \bm{\chi}, \bm{\theta}; \bm{\a})\frac{x_{1}}{\a_{1}} V.
		\end{align*}
		Using \eqref{def_F_chi} and the equation $V = \frac{G(\s)}{\s \log{x_{1}}} (x_{1} / \a_{1})^{\frac{1}{\s} - 1}$,
		we also have
		\begin{align}
			&\frac{G(\s) F_{\bm{\chi}, \bm{\theta}, \s}(\bm{x})}{\log{\| \bm{x} \|}}\\
			&= \l(\frac{1}{\phi(d)}\sum_{u \in (\ZZ / d \ZZ)^{\times}}\bigg| \sum_{\ell = 1}^{r}
			\a_{\ell} e^{-i\theta_\ell}\chi_{\ell}(u) \bigg|^{\frac{1}{\s}}\r)
			\times \frac{G(\s)}{\log{x_{1}}}\l( \frac{x_{1}}{\a_{1}} \r)^{1/\s}\l( 1 + O_{\s,\bm{\chi}, \bm{\a}}\l( \frac{1}{\log{x_{1}}} \r) \r)\\
			\label{pDPJEVex1}
			&= \s \xi(\s, \bm{\chi}, \bm{\theta}; \bm{\a}) \frac{x_{1}}{\a_{1}} V\l(1 + O_{\s,\bm{\chi}, \bm{\a}}\l( \frac{1}{\log{x_{1}}} \r)\r).
		\end{align}
		Hence, the middle in \eqref{nDPJEV3} is equal to
		\begin{align*}
			\exp\l(-(1 - \s)\xi(\s, \bm{\chi}, \bm{\theta}; \bm{\a})\frac{x_{1}}{\a_{1}} V\l( 1 + O_{\s, \bm{\chi}, \bm{\a}}\l( \delta \r) \r)\r),
		\end{align*}
		which combined with the inequalities \eqref{nDPJEV3} yields that
		\begin{align*}
			&\tilde\Psi_{\mathcal{A}}(T, (V, \dots, V), X)\\
			&= \exp\l(-(1 - \s)\xi(\s, \bm{\chi}, \bm{\theta}; \bm{\a})\frac{x_{1}}{\a_{1}} V
			\l( 1 + O_{\s, \bm{\chi}, \bm{\theta}, \bm{\a}}\l( \l(\frac{\log{\log{\|\bm{x}\|}}}{\log{\|\bm{x}\|}}\r)^{1 - \s} \r) \r)\r).
		\end{align*}
		Using \eqref{DPJEV4}, we obtain
		\begin{align*}
			&\tilde\Psi_{\mathcal{A}}(T, (V, \dots, V), X)\\
			&= \exp\l(-\xi(\s, \bm{\chi}, \bm{\theta}; \bm{\a})A(\s)V^{\frac{1}{1-\s}} (\log{V})^{\frac{\s}{1-\s}}
			\l(1 + O_{\s, \bm{\chi}, \bm{\theta}, \bm{\a}}\l(\l(\frac{\log{\log{V}}}{\log{V}}\r)^{1 - \s}\r)\r) \r).
		\end{align*}
		By this equation and Lemma \ref{ESAEG_JEV}, we obtain that when $a_{2}$ is suitably small,
		\begin{align*}
			&\Psi(T, \bm{V}, \bm{\chi}, \bm{\theta}; \bm{\a})
			= \tilde\Psi_{\mathcal{A}}(T, (V, \dots, V), X) + O\l(\frac{1}{T}\meas([T, 2T] \setminus \mathcal{A})\r)\\
			&= \exp\l(-\xi(\s, \bm{\chi}, \bm{\theta}; \bm{\a}) A(\s) V^{\frac{1}{1-\s}} (\log{V})^{\frac{\s}{1-\s}}
			\l(1 + O_{\s, \bm{\chi}, \bm{\theta}, \bm{\a}}\l( \l(\frac{\log{\log{V}}}{\log{V}}\r)^{1-\s} \r) \r) \r)
		\end{align*}
		for $\bm V=(\Xi_1(\sigma, \bm \chi, \bm \theta;\bm \alpha)V, \dots, \Xi_r(\sigma, \bm \chi, \bm \theta; \bm \alpha)V).$
		This completes the proof of Proposition \ref{Main_Prop_JEV}.
	\end{proof}

\section{\textbf{Value distribution of Dirichlet $L$-functions in the strip $\frac{1}{2} < \s < 1$}}

	\begin{proof}[Proof of Theorem \ref{Main_Thm_EV_R}]
		Let $\frac{1}{2} < \s < 1$, $\bm{\theta} \in \RR^{r}$,
		and let $\bm{\chi} = (\chi_{1}, \dots, \chi_{r})$ be an $r$-tuple of Dirichlet characters modulo $q_{j}$.
		Assume that there exists an $\bm{\a} = (\RR_{> 0})^{r}$ such that
		$\Xi_{1}(\s, \bm{\chi}, \bm{\theta}; \bm{\a}), \dots, \Xi_{r}(\s, \bm{\chi}, \bm{\theta}; \bm{\a}) > 0$.
		Let $T$ be a large parameter, and let $X = (\log{T})^{L}$ with $L = \frac{10}{2\s - 1}$.
		Moreover, let $\bm{V} = (\Xi_{1}(\s, \bm{\chi}, \bm{\theta}; \bm{\a})V, \dots, \Xi_{r}(\s, \bm{\chi}, \bm{\theta}; \bm{\a})V)$
		with $V$ a large parameter satisfying $V \leq a_{1} \frac{(\log{T})^{1 - \s}}{\log{\log{T}}}$.
		Here, we choose $a_{1}(\s, \bm{\a}) = a_{2}(\s, L, \bm{\a}) > 0$ where $a_{2}$ is the same constant as in Proposition \ref{Main_Prop_JEV}.
		Using Lemma \ref{KLST}, we see that
		\begin{align} \label{pMTh1}
			& \max_{1 \leq j \leq r}\frac{1}{T}\int_{T}^{2T}
			\bigg| \log{L(\s + it, \chi_{j})} - \sum_{2 \leq n \leq X}\frac{\Lam(n) \chi_{j}^{\star}(n)}{n^{\s+it} \log{n}}
			- \sum_{p \mid q_{j}} \log\l( 1 - \frac{\chi_{j}^{\star}(p)}{p^{\s + it}} \r) \bigg|^{2k}dt\\
			&\leq A^{k} k^{4k}T^{(1 - 2\s)\delta_{\bm{\chi}}}
			+ A^{k} k^{k} \l(\frac{X^{1 - 2\s}}{\log{X}}\r)^{k}
			\leq A^{k} k^{4k}T^{(1 - 2\s)\delta_{\bm{\chi}}} + \l(\frac{k}{(\log{T})^{9}}\r)^{k}
		\end{align}
		for $1 \leq k \leq \frac{1}{L}\frac{\log{T}}{\log{\log{T}}}$ with $\delta_{\bm{\chi}}$ and $A = A(\s, \bm{\chi})$ some positive constants.
		We see that
		\begin{align*}
			\bigg| \sum_{2 \leq n \leq X}\frac{\Lam(n) \chi_{j}(n)}{n^{\s+it} \log{n}}
			- \sum_{2 \leq n \leq X}\frac{\Lam(n) \chi_{j}^{\star}(n)}{n^{\s+it} \log{n}} \bigg|
			= \bigg| \sum_{\substack{p^{\ell} \leq X \\ p \mid q_{j}, \ell \geq 1}}\frac{\chi_{j}^{\star}(p^{\ell})}{\ell p^{\ell(\s+it)}} \bigg|
			\leq \sum_{p \mid q_{j}}\frac{1}{p^{\s}  - 1},
		\end{align*}
		and that
		\begin{align*}
			\bigg| \sum_{2 \leq n \leq X}\frac{\Lam(n) \chi_{j}(n)}{n^{\s+it} \log{n}} - \sum_{p \leq X}\frac{\chi_{j}(p)}{p^{\s + it}} \bigg|
			\leq \sum_{p^{\ell}, \ell \geq 2}\frac{|\chi_{j}(p^{\ell})|}{\ell p^{\ell \s}}.
		\end{align*}
		Applying inequality \eqref{pMTh1} with $k = \lfloor c \frac{\log{T}}{\log{\log{T}}} \rfloor$, $c = \min\{1, \delta_{\bm{\chi}}\}(2\s - 1) / 20$ and the above two inequalities,
		we can find that there exists a set $\mathcal{C} \subset [T, 2T]$ such that
		$\frac{1}{T}\meas([T, 2T] \setminus \mathcal{C}) \leq T^{-d}$ with $d = d(\s, \bm{\chi})$ some positive constant,
		and for all $t \in \mathcal{C}$ and $j = 1, \dots, r$,
		\begin{gather}  \label{def_Set_C}
			\bigg| \log{L(\s + it, \chi_{j})} - \sum_{p \leq X}\frac{\chi_{j}(p)}{p^{\s+it}} \bigg|
			\leq 1 + C_{j},
		\end{gather}
		where $C_{j}$ is the positive constants given by $C_{j} = \sum_{p \mid q_{j}}\frac{1}{p^{\s} - 1} + \sum_{p^{\ell}, \ell \geq 2}\frac{|\chi_{j}(p^{\ell})|}{\ell p^{\ell \s}}$.
		In particular, we have
		\begin{align} \label{ESTCC}
			\frac{1}{T}\meas([T, 2T] \setminus \mathcal{C})
			\leq \exp\l(-2\xi(\s, \bm{\chi}, \bm{\theta}; \bm{\a}) A(\s)V^{\frac{1}{1-\s}}(\log{V})^{\frac{\s}{1-\s}} \r)
		\end{align}
		when $T$ is sufficiently large.
		By the inequality \eqref{def_Set_C} we also find that
		\begin{align*}
			&\tilde\Psi_{\mathcal{C}}(T, (V + (1 + C_{1})/\Xi_{1}(\s, \bm{\chi}, \bm{\theta}; \bm{\a}),
			\dots, V + (1 + C_{r}) / \Xi_{r}(\s, \bm{\chi}, \bm{\theta}; \bm{\a})), X)\\
			&\leq \frac{1}{T}\meas\set{ t \in \mathcal{C}}
			{\Re e^{-i\theta_{j}}\log{L(\s + it, \chi_j)} > \Xi_{j}(\s, \bm{\chi}, \bm{\theta}; \bm{\a})  V_j \text{ for } j = 1, \dots, r}\\
			&\leq \tilde\Psi_{\mathcal{C}}(T, (V - (1 + C_{1}) / \Xi_{1}(\s, \bm{\chi}, \bm{\theta}; \bm{\a}),
			\dots, V - (1 + C_{r}) / \Xi_{r}(\s, \bm{\chi}, \bm{\theta}; \bm{\a})), X)
		\end{align*}
		where $\tilde\Psi_{\mathcal{C}}(T, \bm{V}, X)$ is defined by \eqref{def_Psi_S}.
		By these inequalities and Proposition \ref{Main_Prop_JEV}, we have
		\begin{align*}
			&\Psi(T, \bm{V}, \bm{\chi}, \bm{\theta})\\
			&= \exp\l( -\xi(\s, \bm{\chi}, \bm{\theta}; \bm{\a}) A(\s) V^{\frac{1}{1-\s}}(\log{V})^{\frac{\s}{1-\s}}
			\l( 1 + O_{\s, \bm{\chi}, \bm{\theta}, \bm{\a}}\l( \l(\frac{\log{\log{V}}}{\log{V}}\r)^{1-\s} \r) \r) \r)\\
			&\qquad+ O\l( \frac{1}{T}\meas([T, 2T] \setminus \mathcal{C}) \r)
		\end{align*}
		which together with the inequality \eqref{ESTCC} yields Theorem \ref{Main_Thm_EV_R}.
	\end{proof}

	Next we prove Theorem \ref{thm2}. We fist prove the positivity of $\Xi_i(\sigma, \bm \chi, \bm 1)$ when $r=2$.

	\begin{lemma} \label{RL_X_x_2}
		Let $\bm{\chi} = (\chi_{1}, \chi_{2})$ with $\chi_{1}\not\sim \chi_{2}$ and
		let $\bm \theta=(\theta_1, \theta_2) \in \RR^{2}$.
		Then we have $\Xi_{1}(\s, \bm \chi, \bm \theta; \bm 1) = \Xi_{2}(\s, \bm \chi, \bm \theta;\bm 1) = \xi(\s, \bm{\chi}, \bm \theta;\bm 1) / 2>0$.
	\end{lemma}

	\begin{proof}
		First, we prove $\xi(\s, \bm{\chi}, \bm{\theta}; \bm{1}) > 0$.
		Since $\chi_1$ and $\chi_2$ are inequivalent, we see that
		\begin{align}
			\xi(\tfrac{1}{2}, \bm{\chi}, \bm{\theta}; \bm{1})
			= \frac{1}{\phi(d)}\sum_{u \in (\ZZ / d\ZZ)^{\times}}\sum_{1 \leq j, k \leq 2}e^{-i(\theta_{j} - \theta_{k})}\chi_{j}\bar{\chi_{k}}(u)=2.
		\end{align}
		Therefore, there exists some $u_{0} \in (\ZZ / d\ZZ)^{\times}$ such that $\sum_{j = 1}^{2}e^{-i\theta_{j}}\chi_{j}(u_{0}) \not= 0$, which yields the positivity of $\xi(\s, \bm{\chi}, \bm{\theta}; \bm{1})$.

		Next, we show that
		$\Xi_{1}(\s, \bm{\chi}, \bm \theta; \bm 1) = \Xi_{2}(\s, \bm{\chi}, \bm \theta;\bm 1) = \xi(\s, \bm{\chi}, \bm{\theta}; \bm{1})/2$ when $r = 2$.
		The second equation follows from the first one since we have
		\begin{align*}
			\xi(\s, \bm{\chi}, \bm{\theta}; \bm 1) = \sum_{j = 1}^{r}\Xi_{j}(\s, \bm{\chi}, \bm{\theta}; \bm 1)
		\end{align*}
		by \eqref{Re_xX}.
		If $\chi_{1}(u) \not= 0$, $\chi_{2}(u) \not= 0$, we see that
		\begin{align*}
			&\Re \l(\overline{e^{-i\theta_1}\chi_{1}(u)}(e^{-i\theta_1}\chi_{1}(u) + e^{-i\theta_2}\chi_{2}(u))\r)
			= \Re \l(\overline{e^{-i\theta_2}\chi_{2}(u)}(e^{-i\theta_1}\chi_{1}(u) + e^{-i\theta_2}\chi_{2}(u))\r).
		\end{align*}
		Thus
		\begin{align*}
			&\Xi_{1}(\s, \bm{\chi},\bm \theta; \bm 1)
			= \frac{1}{\phi(d)}\sideset{}{^{*}}\sum_{u \in (\ZZ / d \ZZ)^{\times}}
			\bigg| \sum_{\ell = 1}^{2}e^{-i\theta_\ell}\chi_{\ell}(u) \bigg|^{\frac{1}{\s} - 2}
			\Re \l(\overline{e^{-i\theta_1}\chi_{1}(u)}\sum_{k = 1}^{2}e^{-i\theta_k}\chi_{k}(u)\r)\\
			&= \frac{1}{\phi(d)}\sideset{}{^{*}}\sum_{u \in (\ZZ / d \ZZ)^{\times}}
			\bigg| \sum_{\ell = 1}^{2}e^{-i\theta_\ell}\chi_{\ell}(u) \bigg|^{\frac{1}{\s} - 2}
			\Re\overline{e^{-i\theta_2}\chi_{2}(u)}\sum_{k = 1}^{2}e^{-i\theta_1}\chi_{k}(u)
			= \Xi_{2}(\s, \bm{\chi}, \bm \theta;\bm 1).
		\end{align*}
	\end{proof}

	\begin{remark}
		Note that the proof of Lemma \ref{RL_X_x_2} can be generalized when  $\{\overline{\chi_i}\chi_k\}_{k=1, \dots, \phi(q)}$ runs over all characters modulo $d$.
		We also note that with the choice $\bm \theta=(\theta, \dots, \theta)$ and $\bm \alpha=(1, \dots, 1)$
		\begin{align}
			\xi(\sigma, \bm \chi, \bm \theta; \bm 1)=\phi(q)^{\frac{1}{\sigma}-1}.
		\end{align}
		by orthogonality of Dirichlet characters.
		By the choice of $\bm \chi$, we see as in the proof of Lemma \ref{RL_X_x_2} that
		\begin{align}
			\Xi_j(\sigma, \bm \chi, \bm \theta; \bm{1})=\xi(\sigma, \bm \chi, \bm \theta; \bm 1)/\phi(q)>0.
		\end{align}
	\end{remark}

	\begin{proof}[Proof of Theorem \ref{thm2}]
		Combining Theorem \ref{Main_Thm_EV_R} and Lemma \ref{RL_X_x_2}, we derive Theorem \ref{thm2}.
	\end{proof}

	Next we give the proof of Theorem \ref{DP_DC} which shows that distinct Dirichlet characters ``repels" each other.

	\begin{proof}[Proof of Theorem \ref{DP_DC}]
		From Theorem \ref{thm2} and the identity $$\xi(\s, (\chi_1, \chi_2), (\theta_1, \theta_2); \bm 1) = \xi(\s, (1,\overline{\chi_{1}}\chi_{2}),(0, \theta_2-\theta_1); \bm 1),$$
		it is enough show that for any non-principal character $\chi \pmod q$ and $\theta \in \mathbb{R}$
		\begin{align}	\label{SIDP}
			0<	\frac{1}{\phi(q)}\sum_{a \in (\ZZ / q \ZZ)^{\times}}|1 + e^{-i\theta}\chi(a)|^{1/\s} < 2^{1 / 2\s}.
		\end{align}
		The positivity follows from Lemma \ref{RL_X_x_2}.
		Next we consider the upper bound.
		We have $|1 + e^{-i\theta}\chi(a)|^{\frac{1}{\s}} = 2^{\frac{1}{2\s}}(1 + \Re e^{-i\theta}\chi(a))^{\frac{1}{2\s}}$,
		and $(1 + x)^{1 / 2\s} \leq 1 + \frac{1}{2\s}x$ for $-1 \leq x \leq 1$.
		Using these, we have
		\begin{align*}
			\frac{1}{\phi(q)}\sum_{a \in (\ZZ / q \ZZ)^{\times}}|1 + e^{-i\theta}\chi(a)|^{1/\s}
			\leq \frac{2^{1/2\s}}{\phi(q)}\sum_{a \in (\ZZ / q \ZZ)^{\times}}\l(1 + \frac{1}{2\s} \Re e^{-i\theta}\chi(a)\r)
				= 2^{1 / 2\s},
		\end{align*}
		which proves inequality \eqref{SIDP}.
		Thus, we obtain Theorem \ref{DP_DC} using Theorem \ref{thm2}.
	\end{proof}

	As mentioned in Remark \ref{RMKXi}, the choice $\bm \alpha=\bm 1$ cannot guarantee
	that $\Xi_j(\sigma, \bm \chi, \bm \theta;\bm 1)>0$ for all $\bm \chi, \bm \theta$ and all $\frac{1}{2}< \sigma<1$.
	Nevertheless we can find a particular choice for $\bm \alpha$ such that $\Xi_j(\sigma, \bm \chi, \bm \theta;\bm \alpha)>0$.
	To state the following results, we
	introduce the quantity $B(\bm{\chi}, \bm{\theta})$ defined by
	\begin{align}	\label{def_B_chi_theta}
		B(\bm{\chi}, \bm{\theta})
		= \underset{\chi_{j} \sim \chi_{1}^{2}\bar{\chi_{\ell}}}{\sum_{j=2}^r\sum_{\ell=2}^{r}}\cos(2\theta_{1} - \theta_{j} - \theta_{\ell}).
	\end{align}

	\begin{proposition}	\label{Pos_Xi}
		Let $\frac{1}{2} < \s < 1$, $\bm{\theta} \in \RR^{r}$,
		and let $\bm{\chi} = (\chi_{1}, \dots, \chi_{r})$ be an $r$-tuple of Dirichlet characters. Suppose $\chi_i \not\sim \chi_j$ for $i\not=j.$
		For $\bm{\a} = (\a, 1, \dots, 1)$ with $\a \geq  \a_{0}(\sigma, r)$ sufficiently large, we have
		\begin{align}
			&\Xi_{j}(\s, \bm{\chi}, \bm{\theta}; \bm{\a})\\
			&= \l\{
			\begin{array}{ll}
				\a^{\frac{1}{\s} - 1}\l( 1 + \frac{r - 1 - (2\s - 1)B(\bm{\chi}, \bm{\theta})}{4 \s \a^{2}}\l( \frac{1}{\s} - 2 \r)
				+ O_{\sigma,r}(\a^{-3}) \r)                                       & \text{if\; $j = 1$,}                                                                      \\
				\a^{\frac{1}{\s} - 2}\l(\frac{1 -(2\s - 1) \cos(2\theta_{1} - \theta_{j} - \theta_{\ell})}{2\s}
				+ O_{\sigma,r}(\a^{-1})\r)                                        & \text{if\; $\chi_{j} \sim \chi_{1}^{2} \bar{\chi_{\ell}}$ for some $2 \leq \ell \leq r$,} \\
				\a^{\frac{1}{\s} - 2}\l(\frac{1}{2\s} + O_{ \sigma,r}(\a^{-1})\r) & \text{otherwise.}
			\end{array}
			\r.
		\end{align}
		In particular, we have $\Xi_{j}(\s, \bm{\chi}, \bm{\theta}; \bm{\a})>0$ for all $j$
		when $\bm{\a} = (\a, 1, \dots, 1)$ for $\a \geq \a_{1}(\s, r)$  sufficiently large depending only on $\s$ and $r$.
	\end{proposition}

	\begin{proof}
		Let $\a_{1} = \a$ and $\a_{2} = \cdots = \a_{r} = 1$.
		It follows from the Taylor expansion that
		\begin{align}
			&\bigg| \sum_{\ell = 1}^{r}\a_{\ell} e^{-i\theta_{\ell}} \chi_{\ell}(u) \bigg|^{\frac{1}{\s} - 2}
			= \a^{\frac{1}{\s} - 2}\bigg| 1 + \frac{1}{\a}\sum_{\ell = 2}^{r}e^{i(\theta_{1} - \theta_{\ell})}
			\bar{\chi_{1}}\chi_{\ell}(u) \bigg|^{\frac{1}{\s} - 2}\\
			&= \a^{\frac{1}{\s} - 2}\Biggl\{ 1 + \frac{1}{\a}\l(\frac{1}{\s} - 2\r)
			\Re\sum_{\ell = 2}^{r} e^{i(\theta_{1} - \theta_{\ell})} \bar{\chi_{1}}\chi_{\ell}(u)
			- \frac{1}{2\a^{2}}\l( \frac{1}{\s} - 2 \r)\Re\l(\sum_{\ell = 2}^{r} e^{i(\theta_{1} - \theta_{\ell})} \bar{\chi_{1}}\chi_{\ell}(u) \r)^{2}\\
			&\qqquad+ \frac{1}{2\a^{2}}\l( \frac{1}{\s} - 2 \r)^{2}
			\l( \Re\sum_{\ell = 2}^{r} e^{i(\theta_{1} - \theta_{\ell})} \bar{\chi_{1}}\chi_{\ell}(u) \r)^{2}\Biggr\}
			+ O_{ \sigma,r}\l(\a^{\frac{1}{\s} - 5}\r).
		\end{align}
		By this formula and the fact that $\sum_{\ell = 1}^{r}\a_{\ell} e^{-i\theta_{\ell}} \chi_{\ell}(u)$
		is not zero for all $u \in (\ZZ / d\ZZ)^{\times}$ when $\a$ is large enough (e.g. $\alpha > r-1$),
		we find that
		\begin{align}
			&\Xi_{j}(\s, \bm{\chi}, \bm{\theta}; \bm{\a}) \\
			&=\frac{\a^{\frac{1}{\s} - 2}}{\phi(d)}\sum_{u \in (\ZZ / d\ZZ)^{\times}}
			\Biggl\{ 1 + \frac{1}{\a}\Big(\frac{1}{\s} - 2\Big)
			\Re\sum_{\ell = 2}^{r} e^{i(\theta_{1} - \theta_{\ell})} \bar{\chi_{1}}\chi_{\ell}(u)\\
			&\qqquad- \frac{1}{2\a^{2}}\Big( \frac{1}{\s} - 2 \Big)\Re\Big(\sum_{\ell = 2}^{r} e^{i(\theta_{1} - \theta_{\ell})} \bar{\chi_{1}}\chi_{\ell}(u) \Big)^{2}\\
			&\qqquad+ \frac{1}{2\a^{2}}\Big( \frac{1}{\s} - 2 \Big)^{2}
			\Big( \Re\sum_{\ell = 2}^{r} e^{i(\theta_{1} - \theta_{\ell})} \bar{\chi_{1}}\chi_{\ell}(u) \Big)^{2}\Biggr\}
			\Re \bar{e^{-i\theta_{j}}\chi_{j}(u)}\sum_{k = 1}^{r}\a_{k} e^{-i\theta_{k}} \chi_{k}(u)\\
			&\qqqquad+ O_{\sigma,r}\l( \a^{\frac{1}{\s} - 4} \r)\\
			&= \frac{\a^{\frac{1}{\s} - 2}}{\phi(d)}\sum_{u \in (\ZZ / d\ZZ)^{\times}}
			\l(A_{j,1}(u) + \frac{1}{\a}\Big( \frac{1}{\s} - 2 \Big)A_{j,2}(u) - \frac{1}{2\a}\Big( \frac{1}{\s} - 2 \Big)A_{j,3}(u)
			+ \frac{1}{2\a}\Big( \frac{1}{\s} - 2 \Big)^{2}A_{j,4}(u)\r)\\
			&\qquad+ O_{\sigma,r}\l( \a^{\frac{1}{\s} - 4} \r),
		\end{align}
		where
		\begin{align}
			A_{j,1}(u) &\ceq \Re \bar{e^{-i\theta_{j}}\chi_{j}(u)}\sum_{k = 1}^{r}\a_{k} e^{-i\theta_{k}} \chi_{k}(u),\\
			A_{j,2}(u) &\ceq \sum_{\ell = 2}^{r}\sum_{k = 1}^{r}\a_{k}
			\Re\l(e^{i(\theta_{1} - \theta_{\ell})} \bar{\chi_{1}} \chi_{\ell}(u)\r)
			\Re\l(e^{i(\theta_{j} - \theta_{k})} \bar{\chi_{j}} \chi_{k}(u)\r),\\
			A_{j,3}(u)& \ceq \Re\l( e^{-i(\theta_{1} - \theta_{j})}\chi_{1} \bar{\chi_{j}}(u) \r)
			\Re\l(\sum_{\ell = 2}^{r} e^{i(\theta_{1} - \theta_{\ell})} \bar{\chi_{1}}\chi_{\ell}(u) \r)^{2},\\
			A_{j,4}(u)& \ceq \Re\l( e^{-i(\theta_{1} - \theta_{j})}\chi_{1} \bar{\chi_{j}}(u) \r)
			\l(\Re\sum_{\ell = 2}^{r} e^{i(\theta_{1} - \theta_{\ell})} \bar{\chi_{1}}\chi_{\ell}(u) \r)^{2}.
		\end{align}
		Since $\chi_{1}, \dots, \chi_{r}$ are pairwise inequivalent, there exists one $\chi_{i}$ such that $\bar{\chi_{j}} \chi_{i}$ is principal,
		and there exists at most one $\chi_{i}$ such that $\chi_{1}^{2} \bar{\chi_{j} \chi_{i}}$ is principal.
		By this observation and the orthogonality of characters, we have
		\begin{align}
			\sum_{u \in (\ZZ / d\ZZ)^{\times}}A_{j,1}(u)
			&= \sum_{u \in (\ZZ / d\ZZ)^{\times}}\Re \bar{e^{-i\theta_{j}}\chi_{j}}(u)\sum_{k = 1}^{r}\a_{k}e^{-i\theta_{k}}\chi_{k}(u)
			\\&= \Re\sum_{k = 1}^{r}e^{i(\theta_{j} - \theta_{k})}\a_{k}\sum_{u \in (\ZZ / d\ZZ)^{\times}}\bar{\chi_{j}}\chi_{k}(u)
			=\alpha_j\phi(d)= \l\{
			\begin{array}{cl}
				\a \phi(d) & \text{if \; $j = 1$,}     \\
				\phi(d)    & \text{if \; $j \not= 1$,}
			\end{array}
			\r.
		\end{align}
		and
		\begin{align}
			&\sum_{u \in (\ZZ / d\ZZ)^{\times}}A_{j,2}(u)\\
			&= \sum_{\ell = 2}^{r}\sum_{k = 1}^{r}\frac{\a_{k}}{2}\sum_{u \in (\ZZ / d\ZZ)^{\times}}
			\Re \l(e^{i(\theta_{1} - \theta_{\ell} + \theta_{j} - \theta_{k})}\bar{\chi_{1}\chi_{j}} \chi_{\ell}\chi_{k}(u)
			+ e^{i(\theta_{1} - \theta_{\ell} - \theta_{j} + \theta_{k})}\bar{\chi_{1} \chi_{k}}\chi_{j}\chi_{\ell}(u) \r)\\
			&= \l\{
			\begin{array}{ll}
				\frac{1}{2}(B(\bm{\chi}, \bm{\theta}) + r - 1)                                            & \text{if\; $j = 1$,} \\
				\frac{\a \phi(d)}{2}(1 + \cos(2\theta_{1} - \theta_{j} - \theta_{\ell})) + O_{r}(\phi(d)) & \text{if
				$\chi_{j} \sim \chi_{1}^{2} \bar{\chi_{\ell}}$ for some $2 \leq \ell \leq r$,}                                   \\
				\frac{\a \phi(d)}{2} + O_{r}(\phi(d))                                                     & \text{otherwise.}
			\end{array}
			\r.
		\end{align}
		Moreover, we also find  when $j = 1$ that
		\begin{align}
			\sum_{u \in (\ZZ / d\ZZ)^{\times}}A_{j,3}(u)
			&= \Re \sum_{j = 2}^{r}\sum_{\ell = 2}^{r}e^{-i(2\theta_{1} - \theta_{j} - \theta_{\ell})}
			\sum_{u \in (\ZZ / d\ZZ)^{\times}}\chi_{1}^{2}\bar{\chi_{j}\chi_{\ell}}(u)
			= \phi(d)B(\bm{\chi}, \bm{\theta}),\\
			\sum_{u \in (\ZZ / d\ZZ)^{\times}}A_{j,4}(u)
			&= \frac{1}{2}\Re \sum_{j = 2}^{r}\sum_{\ell = 2}^{r}\sum_{u \in (\ZZ / d\ZZ)^{\times}}
			\l(e^{-i(2\theta_{1} - \theta_{j} - \theta_{\ell})}\chi_{1}^{2} \bar{\chi_{j} \chi_{\ell}}(u)
			+ e^{i(\theta_{\ell} - \theta_{j})}\chi_{j}\bar{\chi_{\ell}}(u)\r)\\
			&= \frac{1}{2}(B(\bm{\chi}, \bm{\theta}) + r - 1) \phi(d).
		\end{align}
		From these formulas and the trivial bounds $\sum_{u \in (\ZZ / d\ZZ)^{\times}}A_{j, i}(u) \ll_{r} \phi(d)$, $i = 3, 4$
		for $j \not= 1$, we obtain Proposition \ref{Pos_Xi}.
	\end{proof}

	\begin{proof}[Proof of Theorem \ref{GDP_DC}]
		Put $\a_{1} = \a$ and $\a_{2} = \cdots = \a_{r} = 1$ with $\a$ sufficiently large.
		Then Proposition \ref{Pos_Xi} shows that $\Xi_j(\sigma, \bm \chi, \bm \theta;\bm \alpha)>0$ for all $j$.
		By \eqref{Re_xX} and Proposition \ref{Pos_Xi}, we have
		\begin{align}
			\xi(\sigma, \bm \chi, \bm \theta; \bm \alpha)
			&= \sum_{j = 1}^{r} \a_j \Xi_j(\sigma, \bm \chi,  \bm \theta; \bm \alpha)
			= \a^{\frac{1}{\s}} + \frac{1}{4\s^{2}}\l( r - 1 - (2\s - 1)B(\bm{\chi}, \bm{\theta}) \r)\a^{\frac{1}{\s} - 2}
			+ O_{r, \sigma}(\a^{\frac{1}{\s} - 3}).
		\end{align}
		On the other hand, we find by Proposition \ref{Pos_Xi} that
		\begin{align}
			&\sum_{j = 1}^{r}\Xi_{j}(\s, \bm{\chi}; \bm{\a})^{\frac{1}{1 - \s}}\\
			&= \a^{\frac{1}{\s}} + \frac{1 - 2\s}{4\s^{2}(1 - \s)}\l( r - 1 - (2\s - 1)B(\bm{\chi}, \bm{\theta}) \r)\a^{\frac{1}{\s} - 2}
			+ O_{r,\sigma}\l(\a^{\frac{1}{1 - \s}( \frac{1}{\s}-2)} + \a^{\frac{1}{\s} - 3}\r).
		\end{align}
		Note that we have the trivial bound $B(\bm \chi, \bm \theta)\leq r - 1$
		since there is at most one $\ell \in \{ 2, \dots, r \}$ such that $\chi_{j} \sim \chi_{1}^2 \overline{\chi_\ell}$ for each $j$.
		Hence, it holds that $r - 1 - (2\s - 1)B(\bm{\chi}, \bm{\theta}) > 0$ for $\frac{1}{2} < \s < 1$.
		Thus, we obtain
		\begin{align}
			&\xi(\s, \bm{\chi}; \bm{\a}) - \sum_{j = 1}^{r}\Xi_{j}(\s, \bm{\chi}; \bm{\a})^{\frac{1}{1 - \s}}\\
			&= \frac{1}{4\s(1 - \s)}\Big(r - 1 - (2\s - 1)B(\bm{\chi}, \bm{\theta})\Big)\a^{\frac{1}{\s}-2}
			+ O_{\s, r}(\a^{\frac{1}{1 - \s}( \frac{1}{\s}-2)} + \a^{\frac{1}{\s} - 3}) > 0
		\end{align}
		for any $\a \geq \a_{0}(r, \s)$ with $\a_{0}(r, \s)$ sufficiently large depending only on $r$ and $\s$.
		This completes the proof.
	\end{proof}

	\begin{proof}[Proof of Theorem \ref{SimulExtreme}]
		From Theorem \ref{Main_Thm_EV_R} and Proposition \ref{Pos_Xi} we see that there exists $c=c(\sigma, r)$ such that for $V= c\frac{(\log T)^{1-\sigma}}{\log\log T}$ with $T$ sufficiently large depending on $\sigma, \bm \chi, \bm \theta$ we have that
		$$\Psi(T, \bm V, \bm \chi, \bm \theta)>0$$
		for $\bm V=(\Xi_{1}(\sigma, \bm \chi, \bm \theta;\bm \alpha)V, \dots, \Xi_{r}(\sigma, \bm \chi, \bm \theta; \bm \alpha) V)$
		where $\bm{\a} = (\a_{1}(\s, r), 1, \dots, 1)$ and $\a_{1}$ is the same constant as in Proposition \ref{Pos_Xi}.
		This gives the lower bound in Theorem \ref{SimulExtreme}.
	\end{proof}

	\begin{acknowledgment*}
		The authors would like to thank Winston Heap for helpful comments on improving the exposition of the paper and Masahiro Mine for providing valuable remarks on a dependence property of $L$-functions.
		We would also like to thank {\L}ukasz Pa\'nkowski for pointing out a mistake in an earlier version of the paper.
		Finally, the authors express their gratitude to the referee(s) for the careful reading and for providing valuable comments and suggestions.
		The first author was supported by Grant-in-Aid for JSPS Research Fellow (Grant Number: 21J00425, 24K16907).
		The second author was supported by Germany's Excellence Strategy grant EXC-2047/1 - 390685813, DFG grant BL 915/5-1 as well as the Max Planck Institute for Mathematics. 
	\end{acknowledgment*}

\end{document}